\documentclass[a4paper,11pt,reqno]{amsart}
\usepackage{amsfonts,amssymb,amscd,amsmath,latexsym,amsbsy,enumerate,a4wide,
tikz,  ifthen, mathrsfs, mathtools}
\usepackage{appendix}

\usepackage[all]{xy}
\usetikzlibrary{calc}
\usepackage{tikz-cd}
\usepackage{mathrsfs}
\usepackage{subfig}
\usepackage{tikz}

\newtheorem{thm}{Theorem}[section]
\newtheorem{defn}[thm]{Definition}
\newtheorem{lem}[thm]{Lemma}

\newtheorem{prop}[thm]{Proposition}
\theoremstyle{definition}
\newtheorem{example}[thm]{Example}
\newtheorem{rmk}[thm]{Remark}

\numberwithin{equation}{section}

\def\al{\alpha}
\def\be{\beta}

\def\de{\delta}
\def\la{\lambda}
\def\si{\sigma}

\def\vp{\varphi}

\def\La{\Lambda}
\def\Up{\Upsilon}
\def\Ga{\Gamma}

\def\R{\mathbb{R}}
\def\C{\mathbb{C}}
\def\N{\mathbb{N}}
\def\Id{\mathbf{1}}

\def\Pol{\mathrm{Pol}}
\def\Ker{\mathrm{Ker}}

\def\adj{\mathrm{adj}}

\newcommand{\rFs}[5]{\,_{#1}F_{#2} \left( \genfrac{.}{.}{0pt}{}{#3}{#4};#5 \right)}

\newcommand{\D}{\mathcal{D}}

\newcommand{\Mh}{\widehat M}
\newcommand{\Dh}{\widehat D}
\newcommand{\Ph}{\widehat P}

\title{Matrix exceptional Laguerre polynomials}
\author{E. Koelink}
\address{IMAPP, Radboud Universiteit, Nijmegen, The Netherlands}
\email{e.koelink@math.ru.nl}
\author{L. Morey}
\address{FaMAF-CIEM, Universidad Nacional de C\'ordoba, Argentina}
\email{lmorey@unc.edu.ar}
\author{P. Rom\'an}
\address{FaMAF-CIEM, Universidad Nacional de C\'ordoba, Argentina}
\email{pablo.roman@unc.edu.ar}

\begin{document}

\begin{abstract}
We give an analog of exceptional polynomials in the matrix valued setting by considering suitable factorizations of a given second order differential operator and performing Darboux transformations. Orthogonality and density of the exceptional sequence is discussed in detail. We give an example of matrix valued exceptional Laguerre polynomials of arbitrary size.
\end{abstract}

\maketitle


\section{Introduction}
The classical families of orthogonal polynomials, namely Hermite, Laguerre, and Jacobi, have been extensively studied and have found numerous applications in mathematics and physics due to their special property of being eigenfunctions of a second order differential operator. These three families are indeed the unique sequences $(p_n)_n$ of polynomials such that for each $n\in \N_0 = \N\cup\{0\}$, $\deg p_n = n$, $p_n$ is orthogonal with respect to a positive measure and $p_n$ is an eigenfunction of a second order differential operator. 

In the past years, the classical families have been extended in different directions. On one hand, an important extension are the matrix valued orthogonal polynomials (MVOPs). This theory started with the work of Krein in the 1940s \cite{Krein1} and has numerous applications in many branches of mathematics and physics. The study of MVOPs has attracted much attention recently and many of the properties of the scalar orthogonal polynomials have been lifted to the matrix valued setting, see \cite{CanteroMV2005}, \cite{CanteroMV2007}, \cite{AAGMM}, \cite{Duran}, \cite{DuranG1}, \cite{GPT}, \cite{GT}, see also the list of references in \cite{DamanikPS}. In particular, the families of MVOPs which are eigenfunctions of a second order differential operator with matrix valued coefficients have been basically classified in a recent paper by Casper and Yakimov \cite{CaspY}. The classification relies on the study of a pair of isomorphic algebras of differential and difference operators, called the Fourier algebras. However, the construction of explicit examples of MVOPs of arbitrary size with the property of being eigenfunctions of a second order differential operator is not an easy task. A useful tool to build new families from the known ones are the shift operators. This technique has been used successfully for Hermite, Laguerre and Gegenbauer type cases in \cite{IKR2}, \cite{KoelR} and \cite{KdlRR}, see also \cite{EMR} for a similar construction for discrete orthogonal polynomials. 

On the other hand, another recent extension is obtained by relaxing the degree condition of the classical families. These are the so-called exceptional orthogonal polynomials. In analogy with the classical set up, these polynomials form an orthogonal set, are eigenfunctions of a second order differential operator with rational coefficients, but some degrees are missing, i.e., there are a finite number of degrees for which no polynomial eigenfunction exists. The study of these polynomials started in \cite{UKM1} and \cite{UKM2} and continued in, for instance, \cite{UllaM}, \cite{TwoStep}, \cite{ArnoN}, \cite{Liaw}, \cite{Sasa1}, \cite {Sasa2}, \cite{Sasa3}, \cite{Sasa4}, \cite{Duran1}, \cite{Duran2}, \cite{Grandati1}, \cite{HeckmanP}. An important tool that has been used for the study of exceptional polynomials are the Darboux transformations. A classification of exceptional polynomials is given in \cite{X-Bochner} where the authors prove that every exceptional family is obtained by applying a finite sequence of Darboux transformations to the classical ones.

The aim of this work is to link these two extensions of classical orthogonal polynomials. Stated differently, we give an analog of exceptional polynomials in the matrix valued setting by considering suitable factorizations of a given second order matrix differential operator and performing Darboux transformations. The paper is organized in two parts. In part I, we discuss a general construction of matrix valued exceptional polynomials. In part II, applying the general theory we give an example of matrix valued exceptional Laguerre polynomials. 

In Section \ref{sec:2ndoperatorsfacDarbux} we consider a matrix valued second order differential operator $T_0$ acting on functions from the right and an eigenfunction $\phi$ of $T_0$, which is called the seed function. Following \cite{EtinGR} and \cite{GoncV} we get a second order differential operator $T_1$ and a first order  operator $A$ having the intertwining property $T_0 A= A T_1$. In addition, the intertwining operator $A$ is built in such a way that it preserves polynomials and regularity of leading coefficients of polynomials. We show that if the seed function has a scalar eigenvalue $\lambda$, then the operator $T_1$ can be obtained from $T_0$ by performing a Darboux transformation, i.e, the operators decompose as $T_0 = AB + \lambda$ and $T_1 = BA + \lambda$ for some first order operator $B$.

In Section \ref{sec:symmdiffopexceptweights} we consider the case where the operator $T_0$ is symmetric with respect to the matrix valued inner product associated to a matrix weight $W$. We introduce the exceptional matrix weight $\widehat W$ and show that the operator $T_1$ is symmetric with respect to the inner product induced by $\widehat W$. We also discuss self-adjoint extensions of the differential operators.

In Section \ref{sec:matriXOL} we prove the main theorem of the paper. We assume that there is a sequence of matrix orthogonal polynomials $(P_n)_{n\in\N_0}$ in $L^2(W)$ being an eigenfunction of $T_0$ for each $n\in\N_0$ i.e. $P_n\cdot T_0 = \Ga_n \cdot P_n.$ We introduce the sequence of matrix valued exceptional polynomials $(\widehat{P}_n)_n$, and we prove that they are eigenfunctions of $T_1$. Moreover we show that the polynomials $(\widehat{P}_n)_n$ are orthogonal with respect to $\widehat{W}$ and, if the eigenvalue $\lambda\not\in \si(\Ga_n)$ for all $n\in \N_0$, then the squared norm of $\widehat{P}_n$ is an invertible positive definite matrix. In contrast to the scalar case, matrix valued orthogonal polynomials may be eigenfunctions of different second order differential operators. Hence the exceptional polynomials $\widehat{P}_n$ depend on the choice of the second order operator $T_0$ as well as on the seed function $\phi$. Under some assumptions, we give an equivalent condition for the density of the exceptional polynomials. 

In Section \ref{sec:T0-diagonalizable} we study families of matrix valued exceptional polynomials obtained from a second order differential operator which is diagonalizable via a nonconstant matrix. 
We show that these polynomials can be described in terms of the scalar exceptional ones. Finally, in Section \ref{sec:fourier} we extend some of the properties of the Fourier algebras given in \cite{CaspY} for the exceptional weight $\widehat{W}$. 

Next we consider an explicit example of matrix valued exceptional Laguerre polynomials for arbitrary size. In Section \ref{sec:MLaguerreweightT0} we recall the weight and one of the differential operators for the matrix Laguerre polynomials to which we apply the general machinery.
In Section \ref{sec:seedmatrixeigenfunctionsforT0} we determine a suitable seed function for the second order matrix differential operator motivated by the theory of regular singularities in this setting, which we recall briefly in Appendix \ref{sec:appA:MDEsingularities} following \cite{CoddL}. In Sections \ref{sec:intertwingLaguerre} and \ref{sec:MatrixXOLaguerre} we introduce an intertwinning operator $A$ for the Laguerre example preserving polynomials and regularity of leading coefficients of polynomials, the weight for which the matrix exceptional Laguerre polynomials are orthogonal and the sequence of the matrix exceptional Laguerre polynomials. Finally, in Section \ref{sec:numerics} we present some numerical information on the zeros of the matrix exceptional Laguerre polynomials discussed in Part \ref{part:matriXOL}.

\noindent
\textbf{Acknowledgements.} The support of Erasmus+ travel grant is grate\-fully acknowledged. The work of Luc\'ia Morey and Pablo Rom\'an was supported by SeCyTUNC.

\section{Preliminaries: Differential and difference operators}
\label{sec:preliminaries}
We consider a sequence $(Q_n(x))_{n\in \mathbb{N}_0}$, such that for each $n\in \mathbb{N}_0$, $Q_n(x)$ is a matrix valued
polynomial in $x$, not necessarily of degree $n$. Let $\mathcal{M}_N$ be the set of all differential operators of the form 
\begin{equation}
\label{eq:DifferentialOperator}
D=\sum_{j=0}^n \partial_x^j F_j(x), \qquad \partial_x^j := \frac{d^j}{dx^j},
\end{equation}
where $F_j:\mathbb{C}\to M_N(\mathbb{C})$ are rational functions of $x$, i.e. each entry is a rational function, with an action from the right on the the sequence $(Q_n(x))_{n\in \mathbb{N}_0}$ given by
$$Q_n(x) \cdot D  = \sum_{j=0}^n \left(\partial_x^j Q_n(x)\right)\,  F_j(x).
$$
The set $\mathcal{M}_N$ is an algebra with the product $D_2\circ D_1$ acting from the right. Since the action of elements in $\mathcal{M}_N$ is from the right, the action of the product $D_2\circ D_1$ on the polynomials $Q_n$ is given by applying first $D_2$ and then $D_1$. 

Now we consider a left action on the sequence $(Q_n(x))_{n\in \mathbb{N}_0}$ by discrete operators. For $j\in\mathbb{Z}$, let $\delta^{j}$ be the discrete operator which acts on a sequence $A:\mathbb{N}_0 \to M_N(\mathbb{C})$ by
$$(\delta^j \cdot A)(n)=A(n+j),$$
where we take the value of a sequence at a negative integer to be equal to the zero matrix. A discrete operator
\begin{equation}
 \label{eq:DifferenceOperator}
M=\sum_{j=-\ell}^k A_j(n) \delta^j,
 \end{equation}
where $A_{-\ell},\ldots,A_k$ are sequences, acts by
\begin{align*}
M \cdot Q_n(x) &= \sum_{j=-\ell}^k A_j(n) \, \delta^j\cdot Q_n(x) = \sum_{j=-\ell}^k A_j(n) \, Q_{n+j}(x).
\end{align*}
In this formula $Q_j(x)$ for negative values of $j$ has to be interpreted as the zero matrix.
The product of two discrete operators is defined by $ M_1\circ M_2$ acting from the left. In this way, the action of a product $M_1\circ M_2$ on the polynomial $Q_n$ is given by applying $M_2$ first and then applying $M_1$.  The set of all discrete operators of the form \eqref{eq:DifferenceOperator} with this product becomes an algebra that we denote by $\mathcal{N}_N$.

In the rest of the paper we will use the notation $D_2\circ D_1 = D_2D_1$ and $M_1\circ M_2 = M_1M_2$. We observe that 
$(M\cdot Q_n)\cdot D= M\cdot (Q_n\cdot D)$ and we denote this two-side action by $M\cdot Q_n\cdot D$.


\part{General construction of matrix valued exceptional polynomials}\label{part:general}

In the first part of the paper we discuss a general construction of matrix valued exceptional polynomials. The construction is inspired by works on the scalar case, e.g \cite{UllaM}, \cite{TwoStep}, \cite{X-Bochner}. 


\section{Second order operators, factorizations and Darboux transforms}
\label{sec:2ndoperatorsfacDarbux}
In this section we discuss a factorization of a matrix valued second order operator $T_0$ as the 
composition of two first order differential operators. We discuss the 
corresponding Darboux factorization. In this section we work with the differential operators in an algebraic way, and we assume that the functions are sufficiently differentiable. 

\subsection{Seed function and intertwining relations}\label{ssec:seedintertwining}
Let us consider a second order differential operator
\begin{equation}\label{eq:definition_T0}
T_0= \frac{d^2}{dx^2}F_2+\frac{d}{dx}F_1+F_0,
\end{equation}
where $F_1$, $F_2$, and $F_3$ are $N\times N$-matrix valued functions which we view as multiplication operator from the right. We assume that $F_1$, $F_2$, and $F_3$ are rational functions, i.e. each entry is a rational function, this means $T_0\in\mathcal{M}_N$. Observe that $T_0$ acts on row vector valued functions from the right or on matrix valued functions from the right viewed as the action on $N$ row valued functions. Let $\phi$ be a matrix eigenfunction of $T_0$ of matrix eigenvalue 
$\La$, i.e. 
\begin{equation}\label{eq:seedfunctiongeneral}
\phi\cdot T_0=\La \cdot \phi.
\end{equation}
\begin{rmk}
\label{rmk:Mphi}
We observe that given an eigenfunction $\phi$ of the operator $T_0$ and a constant matrix $M$ which commutes with $\La$, the function $M\phi$ is also an eigenfunction of $T_0$, for
$$\phi\cdot T_0=\La \cdot \phi \quad \Rightarrow \quad M\phi\cdot T_0= (M\La \cdot \phi) = \La \cdot (M\phi).$$
\end{rmk}

We assume that $\phi$ is a classical solution on a possibly infinite interval $(a,b)$. This interval is related to the $L^2$-space in Section \ref{sec:symmdiffopexceptweights}. 
In case $\La$ is a diagonal matrix this means that each row vector of $\phi$ is an eigenvector. We assume that $\phi(x)$ is invertible as a $N\times N$-matrix, except 
at possibly a finite set of points $x\in\C$. 
The matrix function $\phi$ is the seed function.
By \cite[Thm. 1.1]{EtinGR}, \cite[Thm. 1]{GoncV}
there exist differential operators $\widehat{A}$, $\widehat{T}_1$ such that the interwining relation
\begin{equation}\label{eq:int-relation}
\widehat{A} \widehat{T}_1= T_0 \widehat{A}
\end{equation}
holds as matrix differential operators acting from the right. 
Moreover, the operators $\widehat{A}$, $\widehat{T}_1$ are given by 
\begin{equation}\label{eq:intertwinerAT1}
\widehat{A}=\frac{d}{dx}-\phi^{-1}\phi', \qquad \widehat{T}_1= \frac{d^2}{dx^2}G_2+\frac{d}{dx}G_1+G_0,
\end{equation}
where the coefficient functions $G_2$, $G_1$ and $G_0$ in  $\widehat{T}_1$ are defined by
\begin{gather*}
G_2(x)=F_2(x), \qquad G_1(x)=F_2'(x)+F_1(x)+[\phi(x)^{-1}\phi'(x),F_2(x)], \\
G_0(x)=F_1'(x)-F_1(x)\phi(x)^{-1}\phi'(x)+2(\phi(x)^{-1}\phi'(x))'F_2(x)+\phi(x)^{-1}\phi'(x)G_1(x)+F_0(x),
\end{gather*}
and the commutator is the standard commutator $[\mathcal{A}(x),\mathcal{B}(x)]=\mathcal{A}(x)\mathcal{B}(x)-\mathcal{B}(x)\mathcal{A}(x)$. 
In order to have $\widehat{T}_1$ in the same class, we assume that $\phi^{-1}\phi'$ is a rational matrix function. 

\begin{rmk}
In \cite[Thm. 1.1]{EtinGR} the intertwining operator $\widehat{A}$ is given as quasideterminant of a Wronski matrix. In our setting, since $\widehat{A}$ acts from the right, the expression of the Wronski matrix is the transpose of that in \cite[Thm. 1.1]{EtinGR}. We get the expression
\begin{equation}
\widehat{A}=\begin{vmatrix}
\phi & \phi'\\
1 & \frac{d}{dx}
\end{vmatrix}_{2,2}.
\end{equation}
\end{rmk}

\begin{rmk}
\label{rmk:eigenfunctions-of-T1}
From the intertwining relation \eqref{eq:int-relation}, we get that if $P$ is a  matrix eigenfunction of $T_0$ with eigenvalue matrix $\Ga$, then $P\widehat{A}=P\cdot \widehat{A}$ is a matrix eigenfunction of $\widehat{T}_1$ for the same matrix eigenvalue: 
\begin{equation*}
P\widehat{A}\cdot\widehat{T}_1=P \cdot T_0\widehat{A}=\Gamma \cdot P\widehat{A}.
\end{equation*}
Similarly, row eigenfunctions of $T_0$ are mapped to row eigenfunctions of 
$\widehat{T}_1$ by $\widehat{A}$ for the same eigenvalue. 
\end{rmk}

\subsection{Seed function with scalar eigenvalue}\label{ssec:scalar-eigenvalue}
We add the extra condition that the 
eigenvalue matrix $\La$ for the seed function in \eqref{eq:seedfunctiongeneral} is a scalar $\la\in\mathbb{R}$ times the identity matrix.
We impose this restriction, since the operators act from the right and the eigenvalues act from the left whereas in the Darboux factorization we incorporate the eigenvalue into the differential operator. So the role of the eigenvalue switches, and so we need to 
have an eigenvalue that commutes with the operator. 
Let us consider the following first order matrix differential operator
\begin{equation}\label{eq:defBhat}
\widehat{B}=\frac{d}{dx}B_1+B_0,
\end{equation}
where
$B_1(x)=F_2(x)$, $B_0(x)=F_1(x)+\phi(x)^{-1}\phi'(x)F_2(x)$ are 
rational matrix functions.

\begin{lem}
\label{lem:TO=AB+la}
Let $\widehat{A}$ be as in \eqref{eq:intertwinerAT1}, and $\widehat{B}$ as in \eqref{eq:defBhat}. Then as differential operators acting from the right we have $T_0=\widehat{A}\widehat{B}+\lambda$ and $\widehat{T}_1=\widehat{B}\widehat{A}+\lambda.$
\end{lem}



\begin{proof}
The proof follows by applying the definition of $\widehat{A}$ and $\widehat{B}$ as
in \eqref{eq:intertwinerAT1}, \eqref{eq:defBhat}. By direct calculation we get 
\begin{multline*}
 \widehat{A}\widehat{B}+\la=\frac{d^2}{dx^2}F_2(x)+\frac{d}{dx}F_1(x)-\phi(x)^{-1}\phi''(x)F_2(x)-\phi(x)^{-1}\phi'(x)F_1(x)+\la =T_0 
\end{multline*}
where in the first equality we use $(\phi(x)^{-1})'=-\phi(x)^{-1}\phi'(x)\phi(x)^{-1}$,  and in the second one that $\phi$ is eigenfunction of $T_0$ of eigenvalue $\la$. In a similar way we get
\begin{equation*}
\widehat{B}\widehat{A}+\la=\frac{d^2}{dx^2}G_2(x)+\frac{d}{dx}G_1(x)+B_0'(x)-B_0(x)\phi(x)^{-1}\phi'(x)+\la
\end{equation*}
using \eqref{eq:intertwinerAT1}, \eqref{eq:defBhat}. Using
again $\phi\cdot T_0=\la\cdot \phi$ we see by a straightforward calculation that 
$B_0'(x)-B_0(x)\phi(x)^{-1}\phi'(x)+\la=G_0(x)$. This 
shows $\widehat{B}\widehat{A}+\la
=\widehat{T}_1$. 
\end{proof}

Let $\Pol$ be the vector space of row polynomials. 
We assume that there exists a matrix polynomial $\Up(x)$  such that the 
operator
\begin{equation}\label{eq:defA}
A=\widehat{A}\Up =\frac{d}{dx} \Up - \phi^{-1}\phi'\Up  
\end{equation}
preserves $\Pol$ and the regularity of the leading coefficients. In particular, we assume that $\phi^{-1}\phi'\Up$ is a 
polynomial. By taking any invertible matrix $M$ in Remark \ref{rmk:Mphi}, we obtain a new seed function $M\phi$ with
eigenvalue $\lambda$. However, it follows directly from the definition \eqref{eq:defA} that the operators associated to
$M\phi$ and $\phi$ are equal. As before, $\Up$ is the operator acting from the right 
as multiplication operator. We assume that the matrix polynomial $\Up$ is invertible except at a finite
set. In Part \ref{part:matriXOL}, $\Up$ is a polynomial times the identity, and 
$\Up^{-1}$ is a rational function times the identity. 
We can rewrite the intertwining relation \eqref{eq:int-relation} in terms of $A$ as 
\begin{equation}\label{eq:v2-int-relation}
AT_1=T_0 A, \qquad T_1=\Up^{-1}\widehat{T}_1\Up
\end{equation}
and consequently $T_1=\Up^{-1}\widehat{B}\widehat{A}\Up+\la$. We denote 
\begin{equation}\label{eq:Bhat}
B=\Up^{-1}\widehat{B}=\frac{d}{dx}\Up^{-1}F_2-\Up^{-1}\Up'\Up^{-1}F_2+\Up^{-1}F_1+\Up^{-1}\phi^{-1}\phi'F_2,
\end{equation}
using the explicit expression for $B_0$ and $B_1$ as in \eqref{eq:defBhat}.
Lemma \ref{lem:TO=AB+la} shows $T_0=AB+\la$ and $T_1=BA+\la$ as matrix differential operators acting from the right.


\section{Symmetric differential operators and exceptional weights}\label{sec:symmdiffopexceptweights}

Recall that a matrix measure $\mu$ taking values in the $N\times N$-positive definite matrices can be written as a matrix function $V$, taking values in 
the positive definite matrices $\tau_\mu$-a.e. times the trace measure $\tau_\mu$, which is a Borel measure on $\R$, see Berg \cite{Berg} for more information. Then the square integrable functions are row vector valued measurable functions $f\colon \R \to \C^N$ satisfying 
$\int_\R f(x) V(x) (f(x))^\ast\, d\tau_\mu(x) <\infty$. After identifying functions 
$\tau_\mu$-a.e. we get the Hilbert space $L^2(\mu)$ of square integrable 
row vector functions. In particular, we assume that 
$\int_\R f(x) V(x) (f(x))^\ast\, d\tau_\mu(x)=0$ implies $f=0$ $\tau_\mu$-a.e. meaning that the kernel of $V$ is trivial $\tau_\mu$-a.e.
We assume that $\tau_\mu$ is absolutely continuous 
with respect to the Lebesgue measure $dx$ on some, possibly infinite, interval $(a,b)$.
Then we put $W= V\frac{d\tau_\mu}{dx}$, and 
this is a positive definite matrix function supported on the interval $(a,b)$.
We assume that $W$ has finite moments of all orders, so that $\Pol \subset L^2(W)$, where we use $L^2(W)=L^2(\mu)$ to emphasise the weight function $W$. 
From now on we assume that matrix entries of such a weight $W$ are $C^2$-functions on 
$(a,b)$. 
Let $\langle \cdot, \cdot \rangle_W$ be the 
corresponding inner product. 
We assume that $T_0=\frac{d^2}{dx^2}F_2+\frac{d}{dx}F_1+F_0$ is a second order differential operator acting from the right 
which is symmetric with respect to $W$ on the polynomials. This means that we assume that $T_0$ is well-defined on $\Pol$ and $p\cdot T_0\in L^2(W)$ for all $p\in \Pol$ and 
\begin{equation}\label{eq:symmetryT0}
\langle p\cdot T_0, q \rangle_W = \langle p, q\cdot T_0 \rangle_W,
\qquad \forall\, p,q\in \Pol.
\end{equation} 
We assume that $\Pol$ is dense in $L^2(W)$, which in the classical $N=1$ case means that we deal with a determinate moment problem \cite{Akhi} and the matrix moment problem is discussed by Berg \cite{Berg}.

\begin{rmk}\label{rmk:symmetryeq}
We assume not only the entries of $W$ being $C^2$ on $(a,b)$, but also the matrix entries of $F_2$ and 
the ones of $F_1$ being $C^2$ on $(a,b)$. Moreover, we assume that 
all terms in \eqref{eq:symmetry1}, \eqref{eq:symmetry2}, and \eqref{eq:symmetry3} are separately integrable on $(a,b)$. 
Then by integration by parts, see \cite[Thm.~3.1]{DuraG}, the operator $T_0$ is symmetric with respect to $W$ if and only if 
$F_2(x)W(x)$ and $(F_2(x)W(x))' - F_1(x)W(x)$ vanish at the endpoints $a$ and $b$ and  the symmetry equations 
\begin{align}
    \label{eq:symmetry1}
    F_2W&=WF_2^\ast,\\
    \label{eq:symmetry2}
    (F_2W)'&= \frac12WF_1^\ast+\frac12F_1W,\\
    \label{eq:symmetry3}
    (F_2W)''&=(F_1W)'-F_0W+WF_0^\ast
\end{align}
hold. If we take the derivative of \eqref{eq:symmetry2}  and subtract \eqref{eq:symmetry3} multiplied by two, we can eliminate the coefficient $F_2$. We obtain
$(WF_1^\ast - F_1W)' = 2(WF_0^\ast-F_0W)$.
\end{rmk}

\begin{defn}\label{defn:exceptional-weight}
The matrix exceptional weight is defined by 
\begin{equation*}
\widehat{W}(x)=\Up(x)^{-1}W(x)F_2(x)^\ast(\Up(x)^{-1})^\ast, \qquad x\in (a,b).
\end{equation*}
\end{defn}

Note that Definition \ref{defn:exceptional-weight} states $\widehat{W}$ is a weight. However, even though  \eqref{eq:symmetry1} 
implies that $\widehat{W}$ is self-adjoint, it is not in general positive definite nor is it clear that all its moments are finite. However, if 
$WF_2^\ast$ is a weight, so is $\widehat{W}$. This happens e.g. in case 
$F_2$ is a nonnegative function times the identity on $(a,b)$ as in Part \ref{part:matriXOL}. So we make the assumption that $\widehat{W}$ is positive definite on the interval $(a,b)$ and that it has finite moments of all order. This last condition also involves the matrix polynomial $\Up$. 

Assuming the additional condition that $\widehat{W}$ is a weight with finite moments, we   want to consider $A$ as an operator from $L^2(W)$ to $L^2(\widehat{W})$ 
and $B$ as an operator from  $L^2(\widehat{W})$ to $L^2(W)$, and we want to relate these
operators to each others adjoint operator, see 
Proposition \ref{prop:-BadjA}. For this we need the analogue of integration by parts for  these operators, cf. Remark \ref{rmk:symmetryeq}, which we prove now. 

\begin{lem}\label{lem:strange-relation}
Let $T_0$ be the operator \eqref{eq:definition_T0} and assume that 
$T_0$ is symmetric with respect to $W$ on $(a,b)$. Assume $\phi$ is a seed function for $T_0$ with scalar eigenvalue $\la\in \mathbb{R}$.
Assume moreover that  
$$
x\mapsto \phi(x)'F_2(x)W(x)\phi(x)^\ast, \qquad x\mapsto \phi(x)F_1(x)W(x)\phi(x)^\ast,
$$
both vanish at the same endpoint $a$ or $b$ for the weight $W$. Then 
\begin{align}
\label{eq:strange-condition-v1}
&(WF_2^\ast)'=WF_1^\ast+WF_2^\ast(\phi^{-1}\phi')^\ast-(\phi^{-1}\phi')WF_2^\ast, \\
\label{eq:strange-condition-v2}
&\quad \phi W(\phi'F_2+\frac{1}{2}\phi F_1)^\ast=(\phi'F_2+\frac{1}{2}\phi F_1)W\phi^\ast 
\end{align}
hold for $x\in (a,b)$. 
\end{lem}

\begin{proof}
We first note that \eqref{eq:strange-condition-v1} and 
\eqref{eq:strange-condition-v2} are equivalent using \eqref{eq:symmetry1}
and \eqref{eq:symmetry2}.
So it suffices to prove \eqref{eq:strange-condition-v2}, which we expand as 
\begin{equation}\label{eq:strange-condition-v3}
\phi W F_2^\ast (\phi')^\ast +\frac{1}{2}\phi W  F_1^\ast \phi^\ast =\phi'F_2 W \phi^\ast +\frac{1}{2}\phi F_1 W \phi^\ast,
\end{equation} 
and by assumption both sides vanish at one of the endpoints $a$ or $b$. All terms in \eqref{eq:strange-condition-v3} are $C^1$ on $(a,b)$, since
we assume $\phi$ to be a classical solution on $(a,b)$.
So it  suffices to show that 
the derivatives on both sides are equal. 
Taking derivatives with respect to $x$ on both sides of \eqref{eq:strange-condition-v3}, 
we see that we need to show
\begin{multline*}
 \phi' W F_2^\ast (\phi')^\ast +
 \phi (W F_2^\ast)' (\phi')^\ast + 
 \phi W F_2^\ast (\phi'')^\ast + 
 \frac{1}{2}\phi' W  F_1^\ast \phi^\ast +
 \frac{1}{2}\phi (W F_1^\ast)' \phi^\ast +
 \frac{1}{2}\phi W  F_1^\ast (\phi')^\ast \\
 = \phi''F_2 W \phi^\ast +
 \phi' (F_2 W)' \phi^\ast + 
 \phi'F_2 W (\phi')^\ast +
\frac{1}{2}\phi' F_1 W \phi^\ast +
\frac{1}{2}\phi (F_1 W)' \phi^\ast +
\frac{1}{2}\phi F_1 W (\phi')^\ast 
\end{multline*}
Observe that the term $\phi' W F_2^\ast (\phi')^\ast$ on the left hand side is equal to $ \phi'F_2 W (\phi')^\ast$ on the right hand side by \eqref{eq:symmetry1}.
Next we use \eqref{eq:symmetry2} in the terms $\phi (W F_2^\ast)'(\phi')^\ast$
and $\phi' (F_2 W)' \phi^\ast$ and \eqref{eq:symmetry3}
in the terms $ \frac{1}{2}\phi (W F_1^\ast)' \phi^\ast$ and $\frac{1}{2}\phi (F_1 W)' \phi^\ast$, to see that we have to show that 
\begin{multline}
\label{eq:strange-proof2}
\phi W (\phi'' F_2+\phi'F_1+\phi F_0) ^\ast + \frac12 \phi F_1 W (\phi')^\ast 
+ \frac12 \phi' W F_1^\ast \phi^\ast + \frac12 \phi (WF_2^\ast)''\phi^\ast \\
= 
(\phi'' F_2+\phi'F_1+\phi F_0)W\phi^\ast + \frac12 \phi'WF_1^\ast \phi^\ast + \frac12 \phi (F_2W)'' \phi^\ast + \frac12 \phi F_1W(\phi')^\ast.
\end{multline}
Using \eqref{eq:symmetry1} and the fact $\phi$ is a seed function for $T_0$ with a real scalar eigenvalue $\la$, it follows that \eqref{eq:strange-proof2} is valid.
\end{proof}

Now we can relate $A$ and $B$. Recall that $A$ has been set up so that it preserves polynomials. By the assumptions of finite moments for the weights, we see that 
$A$ maps $\Pol$ into $L^2(\widehat{W})$, where as before we assume that $\widehat{W}$ is a weight function on $(a,b)$. We make the additional assumption that 
$\Pol$ is dense in $L^2(\widehat{W})$. 

\begin{prop}\label{prop:-BadjA}
Assume that $B$ maps $\Pol$ into $L^2(W)$ and that the boundary condition 
\begin{equation*}
(pWF_2^\ast(\Upsilon^{-1})^\ast q^\ast)\biggm|_a^b=0,\qquad \forall\, p,q\in \Pol
\end{equation*}
holds, then $\langle p\cdot A, q \rangle_{\widehat{W}}=-\langle p, q\cdot B \rangle_W$ 
holds for all $p,q\in \Pol$.
\end{prop}

\begin{proof}
We use the definition of $A$ and perform integration by parts to find
\begin{align*}
\langle p\cdot A, q \rangle_{\widehat{W}}=&\, \int_a^b p'\Upsilon \widehat{W} q^\ast-\int_a^b p\phi^{-1}\phi'\Up\widehat{W}q^\ast \\ 
=&\,(p\Up\widehat{W}q^\ast)\biggm|_a^b-\int_a^b p(\Up\widehat{W}q^\ast)'
-\int_a^b p\phi^{-1}\phi'\Up\widehat{W}q^\ast,
\end{align*}
where we drop the arguments in order to ease notation. 
By Definition \ref{defn:exceptional-weight} and the vanishing of the boundary term
we see that 
\begin{align*}
- \langle p\cdot A, q \rangle_{\widehat{W}}=&\, \int_a^b p
W \Bigl( W^{-1}(\Up\widehat{W}q^\ast)' + W^{-1}\phi^{-1}\phi'\Up\widehat{W}q^\ast\Bigr).
\end{align*}
It remains to show that the term in parentheses equals $(qB)^\ast$. Using 
Definition \ref{defn:exceptional-weight} we see that the adjoint of the term in 
parentheses equals 
\begin{align*}
q' \Up^{-1}F_2 + q \bigl( (\Up^{-1}F_2W)'
W^{-1} + \Up^{-1} F_2 W(\phi^{-1}\phi')^\ast W^{-1}\bigr).
\end{align*}
Comparing with \eqref{eq:Bhat} we see that the term for the derivative is correct, and 
in order to recognise this term as $qB$ we need to show 
\[
 (\Up^{-1}F_2W)'
W^{-1} + \Up^{-1} F_2 W(\phi^{-1}\phi')^\ast W^{-1} = 
-\Up^{-1}\Up'\Up^{-1}F_2+\Up^{-1}F_1+\Up^{-1}\phi^{-1}\phi'F_2.
\]
Writing $(\Up^{-1}F_2W)'
W^{-1} = -\Up^{-1}\Up\Up^{-1}F_2 + \Up^{-1}(F_2W)'W^{-1}$, we see that 
$-\Up^{-1}\Up'\Up^{-1}F_2$ cancels on both sides. Multiplying the remaining equality from the left by $\Up$  and from the right by $W$, we see that we need to show 
\[
F_2'W+F_2W' + F_2W(\phi^{-1}\phi')^\ast = F_1W + \phi^{-1}\phi' F_2W.
\]
In turn this is equivalent to the adjoint of \eqref{eq:strange-condition-v1}. 
\end{proof}

\begin{example}\label{exa:Pearson}
We show that Lemma \ref{lem:strange-relation} is a generalization of the 
Pearson equation. We discuss the case as in \cite{KoelR}, but it can be done for more examples of this nature for either differential or difference operators \cite{KdlRR}, \cite{IKR2}, \cite{EMR}. 
Take $T_0$ the second order difference operator defined in 
\cite[Cor.~6.3]{KoelR}, 
$T_0=\frac{d^2}{dx^2}\Phi(x)^\ast+\frac{d}{dx}\Psi(x)^\ast$, where 
$\Phi$, respectively $\Psi$, is an explicit matrix polynomial of degree $\leq 2$, respectively 
$\leq 1$, see \cite[Prop.~5.1, Prop.~5.2]{KoelR}. For the seed function we take the constant identity matrix
$\phi=P_0^{(\alpha,\nu)}=\Id$ so that $\la=0$. Next we take $\Up=\Id$ and so 
$A=\frac{d}{dx}$ and $B=\frac{d}{dx}\Phi^\ast+\Psi^\ast$, which is denoted by 
$S^{(\al,\nu)}$ in \cite[Prop.~6.1]{KoelR}. In this case, $W$ is defined as in 
\cite{KoelR}, which is recalled in \eqref{eq:defWeightMLaguerrepols}, and  
$\widehat{W}$ is the same weight with $\nu$ replaced by $\nu+1$. Then the result 
of Proposition \ref{prop:-BadjA} is \cite[Prop.~6.1]{KoelR}, and 
\eqref{eq:strange-condition-v1} is the (weak) Pearson equation
$(W\Phi)'=W\Psi$. However, the matrix polynomials of other degree do not satisfy the 
requirement for a seed function, since the eigenvalue matrix is not a multiple
of the identity, see \cite[Cor.~6.3]{KoelR} for the expression. 
\end{example}

\begin{rmk}\label{rmk:Pearson-N=1}
In the scalar case $N=1$, \eqref{eq:strange-condition-v1} reduces to 
$(WF_2)'=WF_1$, which is the (weak) Pearson equation as in Example \ref{exa:Pearson}.
This can be solved as $W=\frac{1}{F_2}\exp\left(\int \frac{F_1}{F_2}\right)$. 
\end{rmk}

So assuming $\Pol$ is dense in $L^2(W)$ and $L^2(\widehat{W})$ and $A$ preserving $\Pol$ and $B$ mapping $\Pol$ into $L^2(W)$ we can view $A$ and $B$ as densely defined  operators
\begin{equation*}
\begin{split}
&A \colon D(A)=\Pol \subset L^2(W) 
\to L^2(\widehat{W}), \\
&B \colon D(B)=\Pol \subset  L^2(\widehat{W})
\to L^2(W). 
\end{split}
\end{equation*}
Then Proposition \ref{prop:-BadjA} shows that the adjoint of $A$ 
is densely defined, and $A$ is  a closable operator. Similarly, $B$ is  a densely defined closable operator. We denote the closures 
by $\overline{A}$ and $\overline{B}$ with respective domain $D(\overline{A})$ 
and $D(\overline{B})$. Moreover, Proposition \ref{prop:-BadjA} shows that 
$\overline{A}\subset -\overline{B}^\ast$ and $\overline{B}\subset -\overline{A}^\ast$.
Next we additionally assume that $T_0$ is defined on $\Pol \subset L^2(W)$  and  
that the operator $T_0 \colon \Pol \subset L^2(W) \to L^2(W)$ is essentially
self adjoint. 
Then we have as inclusions for unbounded 
operators on $L^2(W)$
\begin{equation*}
\begin{split}
T_0  = A\, B +\la \subset \overline{A}\, \overline{B} + \la \subset - \overline{A} \,\overline{A}^\ast + \la
\subset  \overline{B}^\ast\, \overline{A}^\ast + \la \subset (\overline{A}\, \overline{B})^\ast + \la \subset T_0^\ast = \overline{T_0}.
\end{split}
\end{equation*}
Since $\overline{T_0}$ and $- \overline{A}\, \overline{A}^\ast + \la$ are self-adjoint
 for $\la\in \R$, 
see e.g. \cite[Ch.~13]{Rudi}, we get 
\begin{equation}\label{eq:exprforT0}
- \overline{A} \,\overline{A}^\ast + \la
=  \overline{B}^\ast\, \overline{A}^\ast + \la = (\overline{A}\, \overline{B})^\ast + \la = T_0^\ast = \overline{T_0}
\end{equation}
as operators acting from the right on $L^2(W)$. 
Similarly, we have as operators acting from the right on 
$L^2(\widehat{W})$ 
\begin{equation*}
T_1  = B\,A + \la \subset \overline{B}\,  \overline{A}+ \la \subset - \overline{A}^\ast\, \overline{A} + \la \subset  \overline{A}^\ast \overline{B}^\ast + \la \subset (\overline{B}\, \overline{A})^\ast + \la
\end{equation*}
and we define $S_1 = - \overline{A}^\ast\, \overline{A} +\la$ as the appropriate self-adjoint extension of $T_1$ acting from the right. According to Deift \cite[Thm.~2]{Deif}, we have 
$\si(\overline{T_0})\setminus\{\la\} = \si(S_1)\setminus\{\la\}$.  
So the spectrum of the self-adjoint operators 
$\overline{T_0}$ and $S_1$ is the same up to possibly one point. 
Moreover, $\si(\overline{T_0})$ and $\si(S_1)$ are contained in $(-\infty,\la]$.


\section{Matrix valued exceptional polynomials}\label{sec:matriXOL}
Having made the set-up of the Sections \ref{sec:2ndoperatorsfacDarbux} 
and \ref{sec:symmdiffopexceptweights}, we assume that all conditions on these sections are satisfied, namely:
\begin{enumerate}
\item \textbf{On the seed function:} There exists a seed function $\phi$ with eigenvalue $\lambda\in\mathbb{R}$ for the operator $T_0.$ The function $\phi(x)$ is invertible as a $N\times N$-matrix, except 
on a finite set of points $x\in\C$. Additionally $\phi^{-1}\phi'$ is a matrix rational function. 
\item \textbf{On the intertwining operator $A$:} There exists a matrix polynomial $\Upsilon$ such that the intertwining operator $A=\frac{d}{dx}\Upsilon-\phi^{-1}\phi\Upsilon$ preserves Pol and the regularity of leading coefficients. The matrix polynomial $\Upsilon$ is invertible except on a finite set.
\item \textbf{On the weight functions:} The matrix weight functions $W$ and $\widehat{W}$ are positive definite with finite moments of all orders. $\Pol$ is dense in $L^2(W)$ and $L^2(\widehat{W})$.
\item \textbf{On the operator $T_0$:} The operator $T_0$ is defined on $\Pol \subset L^2(W)$  and  $T_0 \colon \Pol \subset L^2(W) \to L^2(W)$ is essentially
self adjoint. 
\item \textbf{On the operator $B$:} The operator $B$ on \eqref{eq:Bhat} maps $\Pol$ into $L^2(W)$ 
\item \textbf{Vanishing and boundary conditions:} The functions $x\mapsto \phi(x)'F_2(x)W(x)\phi(x)^\ast,$ $x\mapsto \phi(x)F_1(x)W(x)\phi(x)^\ast,$ vanish at the same endpoint $a$ or $b$ for the weight $W$. The boundary condition $(pWF_2^\ast(\Upsilon^{-1})^\ast q^\ast)|_a^b=0,\quad \forall\, p,q\in \Pol $ holds.
\end{enumerate}
We additionally assume that there is a 
sequence of matrix orthogonal polynomials $(P_n)_{n\in \N_0}$ in $L^2(W)$, i.e. each row of $P_n$ is a polynomial of degree $n$ in $L^2(W)$, so that $P_n(x) = x^n + \text{lot}$ is a monic polynomial and the orthogonality relations 
\begin{equation}\label{eq:genMVOPforW}
\int_a^b P_n(x) W(x) (P_m(x))^\ast\, dx = \de_{m,n} H_n
\end{equation}
hold. Here the integration is done entrywise, so that the inner product of the $i$-th row of $P_n$ and the $j$-the row of $P_m$ corresponds to the $(i,j)$-entry of 
the right hand side, i.e. $\de_{m,n} (H_n)_{i,j}$. Note in particular, $H_n$ is a
positive definite matrix. So we can also define a matrix valued inner product on 
matrix functions for which each row is contained in $L^2(W)$. The corresponding 
matrix valued inner product is also denoted by $\langle \cdot, \cdot\rangle_W$ 
Moreover, we assume that $P_n$ is an eigenfunction for 
$T_0$ for each $n\in \N_0$, i.e. 
\begin{equation}\label{eq:PneigenfunctionsT0}
P_n\cdot T_0 = \Ga_n \cdot P_n, \qquad \forall \, n\in \N_0.
\end{equation}
Note that this implies that the coefficients $F_i$ of $T_0$ are matrix polynomials of at most degree $i$ for $i\in \{0,1,2\}$ and so $T_0$ preserves $\Pol$. 

\begin{thm}
\label{thm:ortogonalidad-excep}
Consider the monic matrix orthogonal polynomials $(P_n)_{n\in \N_0}$, where the orthogonality is with respect to $W$ as in \eqref{eq:genMVOPforW}. Assume moreover that each $P_n$ is a matrix eigenfunction of $T_0$ as in \eqref{eq:PneigenfunctionsT0}. 
Then the sequence of matrix polynomials $(\widehat{P}_n)_{n\in \N_0}$, where $\widehat{P}_n= P_n\cdot A$, forms 
a set of matrix eigenfunctions for $S_1$. For all $n\in \mathbb{N}_0$, the polynomial  $\widehat{P}_n$ satisfies
\[
\widehat{P}_n  \cdot T_1 = \Ga_n \cdot \widehat{P}_n,
\]
and
\[
\int_a^b \widehat{P}_n(x) \widehat{W}(x) (\widehat{P}_m(x))^\ast\, dx = \de_{m,n} \widehat{H}_n, \qquad \widehat{H}_n= H_n (\la -\Ga_n)^\ast,
\]where $\widehat{H}_n$ is a positive definite matrix. Moreover, if we assume $\la\not\in \si(\Ga_n)$ for all $n\in \mathbb{N}_0$, then $\det \widehat{H}_n \neq 0$.
\end{thm}

\begin{proof} 
Since $S_1$ extends $B\, A+\la$ and $\widehat{P}_n \cdot B= P_n  \cdot AB = (\Ga_n-\la) \cdot P_n$ 
is a polynomial, we see that 
\[
\widehat{P}_n \cdot S_1 = (\Ga_n-\la) \cdot P_n \cdot A + \la \widehat{P}_n = \Ga_n \cdot \widehat{P}_n,
\]
cf. Remark \ref{rmk:eigenfunctions-of-T1}. 
From Proposition \ref{prop:-BadjA} for row vector polynomials, we can extend to 
$\langle P \cdot A, Q\rangle_{\widehat{W}}=-\langle P,Q \cdot B\rangle_W$ for all $P$, $Q$ matrix polynomials. So we get
\begin{equation*}
\begin{split}
\langle \widehat{P}_n,\widehat{P}_m\rangle_{\widehat{W}}\, = &\, \langle P_n \cdot A, P_m \cdot A\rangle_{\widehat{W}}\, =\, -\langle P_n,P_m \cdot AB\rangle_W \\
=&\, \langle P_n, P_m \cdot (\la -T_0)\rangle_W \, =\,\de_{n,m}\mathcal{H}_n(\la-\Ga_n)^\ast. 
\end{split}
\end{equation*}
Note that $H_n (\la-\Ga_n)^\ast = (\la-\Ga_n) H_n$, since we can also take $AB$ in 
the first leg of the inner product. By positivity of the inner product we also see that $H_n (\la-\Ga_n)^\ast$ is a positive
definite matrix. Moreover, if $\la\not\in \si(\Ga_n)$ for all $n\in \mathbb{N}_0$, then $\det \widehat{H}_n \neq 0.$
\end{proof}

In general the degree of $\widehat{P}_n$ is not $n$, but higher since 
$A$ does not in general preserve the degree of the polynomial. 
Moreover, Theorem \ref{thm:ortogonalidad-excep} also does not claim that the rows of 
$(\widehat{P}_n)_{n\in \N_0}$ form a basis for $L^2(\widehat{W})$. 

\begin{rmk}
Observe that the polynomials $\widehat{P}_n$ depend on the second order operator $T_0$, the seed function $\phi$ and the polynomial $\Upsilon.$ We call them matrix valued exceptional polynomials associated to $(T_0,\phi, \Upsilon)$.
\end{rmk}

\begin{prop}\label{prop:densityofexceptpols}
The rows of $\widehat{P}_n$, $n\in \N_0$, are dense in $L^2(\widehat{W})$ if and only if $\la\not\in \si_p(S_1)$. 
\end{prop}

\begin{proof} Consider $A\colon \Pol \subset L^2(W)\to L^2(\widehat{W})$ and assume 
$f\in L^2(\widehat{W})$ is perpendicular to $v \widehat{P}_n$ for all $n\in \N_0$ and all row vectors $v$. Then 
$0 = \langle v \widehat{P}_n, f\rangle_{\widehat{W}} 
= \langle v P_n\cdot A, f\rangle_{\widehat{W}}$, so that this is a continuous linear 
functional on $D(A)=\Pol$. In particular, $f\in D(A^\ast)=D(\overline{A}^\ast)$ and 
$f\, A^\ast=0$. Then $f\in D(S_1)$ and $f\cdot S_1 =\la\cdot f$. So if the rows of
$\widehat{P}_n$, $n\in \N_0$, are not dense there is a non-zero $f$ and $\la\in \si_p(S_1)$. 

Conversely, if $\la\in \si_p(S_1)$ we have $0\not=f \in D(S_1) \subset D(\overline{A}^\ast)$ with $f\cdot S_1= \la\cdot f$ in $L^2(\widehat{W})$, so that in particular $f\, \overline{A}^\ast \overline{A}=0$. Taking inner product with $f\in D(\overline{A}^\ast)$ we obtain 
$\langle f\cdot \overline{A}^\ast, f\cdot\overline{A}^\ast\rangle_{W}=0$ 
and $f\in \Ker(\overline{A}^\ast)$. Then $f$ is perpendicular to all $v\widehat{P}_n$ 
as in the previous paragraph. 
\end{proof}


\section{Construction of exceptional polynomials via a diagonalizable operator $T_0$}
\label{sec:T0-diagonalizable}
In this section we study families of matrix valued exceptional polynomials obtained from a second order differential operator which is diagonalizable via a nonconstant matrix. 

Let $w_1(x), \ldots, w_N(x)$ be a collection of classical scalar weights such that all $w_i$'s belong to one of the families of Jacobi, Laguerre or Hermite weights. For each weight $w_i$, $i=1,\ldots, N$, we take the a symmetric second order differential operator $\mathcal{L}_i$ with respect to $w_i$.

We now consider a weight matrix $W$ having a LDU decomposition 
\begin{equation}
\label{eq:weightstructure5}
W(x)=L(x)D(x)L(x)^\ast,    
\end{equation}
where $L$ is lower triangular $N\times N$ polynomial matrix with $L_{i,i}=1$ for all $i=1,\ldots,N$, and 
$D=\mathrm{diag}(w_1,\ldots, w_N)$. We observe that the structure of $W$ agrees with the Gegenbauer, Laguerre and Hermite type families studied in \cite{KdlRR}, \cite{KoelR}, \cite{IKR2} and \cite{DuranManuel}. The diagonal operator $T^d_0=\mathrm{diag}(\mathcal{L}_1,\ldots, \mathcal{L}_N)$ is symmetric with respect to the diagonal weight matrix $D$. We will write $T^d_0=\frac{d^2}{dx^2}F^d_2+\frac{d}{dx}F^d_1+F^d_0$.

Let $T_0=\frac{d^2}{dx^2}F_2(x)+\frac{d}{dx}F_1(x)+F_0(x)$ be the second order differential operator obtained by 
conjugation of $T^d_0$ by $L$. Then the coefficients $F_i$ and $F^d_i$ are related by:
$$F_2L=LF^d_2,\qquad    F_1L= 2\frac{dL}{dx}F^d_2+L F^d_1, \qquad  F_0L=\frac{d^2L}{dx^2}F^d_2+\frac{dL}{dx}F^d_1+LF^d_0.$$
\begin{rmk}
It follows from \cite[Prop. 4.2]{KdlRR} that $T_0$ is symmetric with respect to $W$ if and only if $T^d_0$ is symmetric with respect to $D$. We observe that given $\phi$ an eigenfunction of $T_0$ for eigenvalue $\Lambda,$ $\phi L$ is an eigenfunction of $T^d_0$ for the same eigenvalue. Similarly, given $\phi^d$ an eigenfunction of $T^d_0,$ $\phi^d L^{-1}$ is eigenfunction of $T_0.$    
\end{rmk}

We assume that all conditions stated at the beginning of Section \ref{sec:matriXOL} are satisfied for the operator $T^d_0$. In particular we assume that there exists $\phi^d$ a diagonal seed function for $T^d_0$ of scalar eigenvalue $\lambda$, and a diagonal matrix polynomial $\Upsilon$, invertible except on a finite set, such that the intertwining operator $A^d=\frac{d}{dx}\Upsilon-(\phi^{d})^{-1}(\phi^d)'\Upsilon$ preserves Pol and preserves the regularity of leading coefficients of polynomials. The operator $T_0^d$ factorizes as $T_0^d=A^dB^d+\lambda,$ where $B^d$ is given in \eqref{eq:Bhat}. The exceptional diagonal weight is given by 
$$\widehat{D}(x)=\Upsilon(x)^{-1}D(x)F^d_2(x)^\ast(\Upsilon(x)^{-1})^\ast.$$

Given that $T_0=LT_0^dL^{-1},$ we can transfer this to the non-diagonal setting. We get 
\begin{equation*}
T_0=(LA^dL^{-1})(LB^dL^{-1})+\lambda. 
\end{equation*}
We assume that $\Upsilon$ is taken in such a way that $A=LA^dL^{-1}$ also preserves Pols. In the case that $\Upsilon$ commutes with $L^{-1}$, the exceptional non-diagonal weight is given by 
$$\widehat{W}(x)=\Upsilon(x)^{-1}W(x)F_2(x)^\ast(\Upsilon(x)^{-1})^\ast=L(x)\widehat{D}(x)L(x)^\ast.$$
\begin{rmk}
Since $\phi^dL^{-1}$ is an eigenfunction of $T_0$ of scalar eigenvalue $\lambda$, it can be used to construct the first order operator given in \eqref{eq:defA}. Explicitly we
obtain
\begin{equation}
    \label{eq:conjugacion1}
\frac{d}{dx}\Upsilon -(L(\phi^d)^{-1}(\phi^d)'L^{-1} - L'L^{-1})\Upsilon.
\end{equation}
On the other hand, 
\begin{equation}
\label{eq:conjugacion2}
LA^dL^{-1} = \frac{d}{dx}L\Upsilon L^{-1} - L(\phi^d)^{-1}(\phi^d)'\Upsilon L^{-1} + L'\Upsilon L^{-1}.
\end{equation}
If $\Upsilon$ commutes with $L^{-1}$, then \eqref{eq:conjugacion1} and \eqref{eq:conjugacion2} coincide.
\end{rmk}
The sequence of monic orthogonal polynomials for the diagonal weight $D$ is given by $P_n^d(x)=\mathrm{diag}(p^1_n(x),\ldots, p^N_n(x))$ where $(p_n^i)_n$ are the monic orthogonal polynomials with respect to $w_i$. On the other hand, the sequence of matrix valued polynomials $Q_n(x) = P_n^d(x)L(x)^{-1}$ is orthogonal with respect to $W$ and each $Q_n(x)$ is an eigenfunction of $T_0$. However, the leading coefficient of $Q_n$ is, in general, not invertible.

By applying the operator $A$ to the sequence $Q_n$ we obtain
$$\widehat Q_n(x)=Q_n(x)\cdot A = (P_n^d(x) L(x)^{-1})\cdot  (L(x)A^dL(x)^{-1}) = (P_n^d(x) \cdot A^d) L(x)^{-1}.$$
By the construction, the sequence $\widehat Q_n(x)$ satisfies the following properties:
\begin{enumerate}
\item $\widehat Q_n$ is an eigenfunction of the Darboux transform of $T_0$, see \eqref{eq:v2-int-relation},
\item $(\widehat Q_n)_n$ forms an orthogonal family with respect to the weight $\widehat W$.
\end{enumerate}
We observe that we cannot guarantee the regularity of the leading coefficients of $\widehat Q_n$. The polynomials $\widehat P_n(x)= P_n(x) \cdot A$, where $(P_n)_n$ are the monic orthogonal polynomials with respect to $W$ also satisfy the properties (1) and (2). The regularity of the leading coefficients is not guaranteed in this case either.


\section{Fourier algebras}
\label{sec:fourier}

Recently, Casper and Yakimov \cite{CaspY} presented a pair of isomorphic algebras, known as the Fourier algebras, which are linked with a sequence of matrix valued orthogonal polynomials for a weight $W$, the so-called Fourier algebras. The goal of this section is to extend some of the properties of the Fourier algebras given in \cite{CaspY} for the exceptional weight $\widehat{W}$.  Next, we adapt the
cons\-truc\-tion of the Fourier algebras given in \cite[Def. 2.20]{CaspY} to our setting:
\begin{defn}
Given a sequence of matrix valued polynomials $(Q_n(x))_{n\in \mathbb{N}_0}$, we define:
\begin{equation}
\label{eq:definition-Fourier-algebras}
\begin{split}
\mathcal{F}_L(Q)&=\{ M\in \mathcal{N}_N \colon \exists \D\in \mathcal{M}_N,\, M\cdot Q = Q\cdot \D \} \subset \mathcal{N}_{N},\\
 \mathcal{F}_R(Q)&=\{ \D\in \mathcal{M}_N \colon \exists M\in \mathcal{N}_N,\, M\cdot Q = Q\cdot \D \}\subset \mathcal{M}_{N}.
\end{split}
\end{equation}
\end{defn}

In this definition, it is important to note that the Fourier algebras are linked to a sequence matrix-valued polynomials $(Q_n)_n$, which is not necessarily the sequence of matrix valued orthogonal polynomials with respect to a weight $W$. 
In particular, if $(P_n(x))_{n\in \mathbb{N}_0}$ is the sequence of MVOPs with respect to a matrix weight, it was established in \cite{CaspY} that the left and right Fourier algebras $\mathcal{F}_L(P)$ and $\mathcal{F}_R(P)$ are isomorphic. We will now show that the same holds true for the exceptional Fourier algebras $\mathcal{F}_L(\widehat P)$ and $\mathcal{F}_R(\widehat P)$. In the following proposition, $\Gamma_n$ denotes the eigenvalue as in \eqref{eq:PneigenfunctionsT0}.

\begin{prop}
\label{prop:psi-well-defined}
For every $\widehat M \in \mathcal{F}_L(\widehat P)$, there exists one and only one $\widehat D \in \mathcal{F}_R(\widehat P)$ such that $\widehat M \cdot \widehat{P}_n(x) = \Ph_n(x)\cdot \Dh$. Conversely, if we assume that $(\Gamma_n-\lambda)$ is invertible for all $n\in\N_0$ then for every $\widehat D \in \mathcal{F}_R(\widehat P)$, there exists a unique $\widehat M \in \mathcal{F}_L(\widehat P)$ such that $\widehat M \cdot \widehat{P}_n(x) = \Ph_n(x)\cdot \Dh$.
\end{prop}

\begin{proof}
In order to prove the first statement, it suffices to show that 
if  $\Dh \in\mathcal{F}_R(\widehat P)$ is such that $\Ph_n(x)\cdot \Dh =0$ for all $n\in \N_0$, then $\Dh=0$. Let us assume that $\widehat D$ is an operator of order $s$ given by
$$\Dh = \sum_{k=0}^s \partial_x^k \, F_k(x).$$
Observe that if $\Ph_n(x)\cdot \Dh =0$ for all $n\in \N_0$, then $P_n(x)\cdot A\Dh =0$ for all $n\in \N_0$. Therefore by \cite{CaspY}, \cite[Lem. 1]{DeanER}, $A\Dh =0$. Using the explicit expression \eqref{eq:defA} we obtain that the coefficient of $\partial_x^{s+1}$ in $A\Dh$ is equal to $\Upsilon(x)F_{s}(x)$ and since $\Upsilon(x)$ is invertible except at a finite set, we get that $F_{s}(x)=0$. Proceeding recursively shows that
$F_k(x) = 0$ for all $k=0,\ldots, s$.

For the second statement it is enough to show that if $\Mh \in \mathcal{F}_L(\widehat P)$, and $\Mh\cdot \Ph_n(x)=0$ for all $n\in \mathbb{N}_0$, then $\Mh =0$. Assume that $\Mh \cdot \Ph_n(x) =0$ for all $n\in \N_0$. Then we have
 $$ 0=\Mh \cdot \Ph_n(x) \quad \Rightarrow \quad 0=\Mh \cdot \Ph_n(x)\cdot B =
 \Mh \cdot P_n(x)\cdot AB = \Mh(\Gamma_n-\lambda) \cdot P_n(x),
$$
for all $n\in \N_0$. Now by \cite{CaspY}, \cite[Lem. 1]{DeanER} we get $\Mh(\Gamma_n-\lambda)=0$ for all $n\in \N_0$, and since $(\Gamma_n-\lambda)$ is invertible for all $n\in\N_0$ we obtain $\Mh=0$.
\end{proof}

For the rest of this section, and as in Theorem \ref{thm:ortogonalidad-excep}, we assume that $\Gamma_n - \lambda$ is an invertible matrix for all $n\in \mathbb{N}_0$. By Proposition \ref{prop:psi-well-defined} there is a well
defined vector space isomorphism $\widehat \psi$ from $\mathcal{F}_L(\widehat P)$ to $\mathcal{F}_R(\widehat P)$, given by
$$\widehat \psi(\Mh) = \Dh,$$
where $\widehat D$ is the unique differential operator in $\mathcal{F}_R(\widehat P)$ such that $\Mh\cdot \widehat{P}_n(x) = \widehat{P}_n(x) \cdot \Dh$. Moreover for all $M_1,M_2 \in \mathcal{F}_L(P)$ we have that
$$\widehat{P}_n\cdot \psi(M_1M_2)  = M_1M_2 \cdot \widehat{P}_n = \widehat{P}_n \cdot \psi(M_2)\psi(M_1).$$
Since differential operators act form the right, we get that $\psi$ is an algebra isomorphism.

\subsection{Relation between the Fourier algebras of standard and exceptional polynomials}
In this subsection we introduce linear maps relating the Fourier algebras for the monic orthogonal polynomials and matrix exceptional polynomials.
\begin{prop}
\label{prop:maps-fourier}
    The following relations hold true:
    \begin{enumerate}
        \itemsep 0.28em
        \item If $\widehat D\in \mathcal{F}_R(\widehat P)$, then  $A\widehat D B \in \mathcal{F}_R(P)$. 
        \item If $D\in \mathcal{F}_R(P)$, then  $BDA\in \mathcal{F}_R(\widehat P)$.
        \item If $\widehat M\in \mathcal{F}_L(\widehat P)$, then  $\widehat M(\Gamma_n -\lambda) \in \mathcal{F}_L(P)$. 
        \item $M\in \mathcal{F}_L(P)$, then  $(\Gamma_n -\lambda)M \in \mathcal{ F}_L(\widehat P)$.
    \end{enumerate}
\end{prop} 

\begin{proof}
We give the proof of (1) and (2). The other statements are proven in a similar way. If $\widehat{D}\in \mathcal{F}_R(\widehat P)$, then there exists $\widehat M$ such
that $\widehat{M} \cdot \widehat P_n(x) = \widehat P_n(x) \cdot \widehat D$. Therefore
\begin{multline*}
P_n(x)\cdot A \widehat D B  = \widehat P_n(x) \cdot \widehat D B  
= \widehat M \cdot \widehat P_n(x) \cdot B  = \widehat M \cdot P_n(x) \cdot AB \\
= \widehat M \cdot P_n(x) \cdot (T_0-\lambda) = \widehat M (\Gamma_n-\lambda) \cdot P_n(x),
\end{multline*}
and $A\widehat D B \in \mathcal{F}_R(P).$ The proof of the second statement is similar
\begin{multline*}
    \widehat P_n(x)\cdot B D A  = P_n(x) \cdot ABDA  
    = P_n(x) \cdot (T_0-\lambda)DA  = (\Gamma_n-\lambda) \cdot P_n(x) \cdot DA \\
    = (\Gamma_n-\lambda)M \cdot P_n(x) A = (\Gamma_n-\lambda)M \cdot \widehat P_n(x).
\end{multline*}
and $BDA\in \mathcal{F}_R(\widehat P)$. This completes the proof of (2).
\end{proof}
By Proposition \ref{prop:maps-fourier}, we have maps $\xi:\mathcal{F}_R(P) \to  \mathcal{F}_R(\widehat P)$ 
and $\widehat \xi: \mathcal{F}_R(\widehat P) \to \mathcal{F}_R(P)$ given by
\begin{equation}
\label{eq:maps-xi}
    \xi(D) = BDA, \qquad \widehat \xi(\widehat D)  = A\widehat{D}B.
\end{equation}
Similarly, we have maps $\chi:\mathcal{F}_L(P) \to  \mathcal{F}_L(\widehat P)$ 
and $\widehat \chi: \mathcal{F}_L(\widehat P) \to \mathcal{F}_L(P)$ given by
\begin{equation}
\label{eq:maps-chi}
    \chi(M) = (\Gamma_n-\la)M, \qquad \widehat \chi(\widehat M)  = \widehat{M}(\Gamma_n-\la)
\end{equation}
We observe that $\xi, \widehat \xi, \chi, \widehat \chi$  are linear maps and are not algebra homomorphisms in general.

\begin{rmk}
We observe that for all $D\in \mathcal{F}_R(P)$ and $\widehat D\in \mathcal{F}_R(\widehat P)$ we have
$$(\widehat\xi \circ \xi)(D) = (T_0-\lambda) D (T_0-\lambda), 
\qquad (\xi\circ \widehat \xi)(\widehat D)  = (T_1-\lambda) \widehat D (T_1-\lambda).$$
Similarly, for all $M\in \mathcal{F}_L(P)$ and $\widehat M\in \mathcal{F}_L(\widehat P)$ we have
$$ (\widehat{\chi}\circ \chi)(M)=(\Gamma_n-\la)M(\Gamma_n-\la) \qquad (\chi\circ \widehat{\chi})(\widehat{M})=(\Gamma_n-\la)\widehat{M}(\Gamma_n-\la) $$
\end{rmk}
\begin{thm}
\label{thm:FourierAlgebrasDiagram}
Let $(P_n)_n$ be the sequence of monic orthogonal polynomials with respect to $W$, $(\widehat P_n)_n$ be a sequence of exceptional polynomials associated to $(T_0,\phi,\Upsilon)$ and let $\xi, \widehat \xi, \chi, \widehat \chi$ be the maps given in \eqref{eq:maps-xi}, \eqref{eq:maps-chi}. Then then following diagrams are commutative:
$$\begin{tikzcd}
	\mathcal{F}_L(P) 
	\arrow[r, "\psi"] 
	\arrow[d, "\chi"]
	& \mathcal{F}_R(P) 
	\arrow[d, "\xi"] \\
	\mathcal{F}_L(\widehat P) 
	\arrow[r, "\widehat \psi"]
	&  \mathcal{F}_R(\widehat P)
\end{tikzcd}, \qquad 
\begin{tikzcd}
	\mathcal{F}_L(P) 
	\arrow[r, "\psi"] 
	& \mathcal{F}_R(P) \\
	\mathcal{F}_L(\widehat P) 
     \arrow[u, "\widehat{\chi}"]
	\arrow[r, "\widehat \psi"]
	&  \mathcal{F}_R(\widehat P)
    \arrow[u, "\widehat{\xi}"]
\end{tikzcd}$$
\end{thm}
\begin{proof}
Let $M\in\mathcal{F}_L(P)$ and $D\in\mathcal{F}_R(P)$ the corresponding differential operator. Observe that
\begin{multline*}
    \chi(M)\cdot \widehat{P}_n=(\Gamma_n-\lambda)M\cdot \widehat{P}_n=(\Gamma_n-\lambda)M\cdot P_n\cdot A=(\Gamma_n-\lambda)\cdot P_n \cdot D A
    \\=P_n \cdot(T_0-\lambda)DA =P_n\cdot ABDA=\widehat{P}_n\cdot BDA= \widehat{P}_n\cdot \xi(\psi(M)).
\end{multline*}
We conclude $\widehat{\psi}(\chi(M))=\xi(\psi(M)),$ and the first diagram is commutative. 

Similarly, let $\widehat{M}\in\mathcal{F}_L(\widehat{P})$ and $\widehat{D}\in\mathcal{F}_R(\widehat{P})$ the corresponding differential operator
\begin{multline*}
    \widehat{\chi}(\widehat{M})\cdot P_n=\widehat{M}(\Gamma_n-\lambda)\cdot P_n=\widehat{M}\cdot P_n\cdot (T_0-\lambda)=\widehat{M}\cdot P_n\cdot AB=\widehat{M}\cdot \widehat{P}_n\cdot B
    \\=\widehat{P}_n \cdot\widehat{D}B =P_n\cdot A\widehat{D}B= P_n\cdot \widehat{\xi}(\widehat{\psi}(\widehat{M})).
\end{multline*}
We conclude $\psi(\widehat{\chi}(\widehat{M}))=\widehat{\xi}(\widehat{\psi}(\widehat{M})),$ and the second diagram is commutative. This completes the proof of the theorem.
\end{proof}

\begin{example}
\label{ex:laguerre-escalar}
\label{example:clasical-laguerre}
The classical Laguerre operator is given by
\begin{equation}
\label{eq:scalar-laguerre-op}
\mathcal{L}^{(\alpha)} = x\frac{d^2}{dx^2} + (\alpha+1-x)\frac{d}{dx}.    
\end{equation}
Seed functions having a polynomial part and eigenvalues
of the Laguerre operator are known, see e.g. \cite[\S 6.1]{Erdelyi}
\begin{enumerate}
   
     \item $f_1(x)=x^{-\alpha}L_m^{(-\alpha)}(x),\quad$ eigenvalue: $\alpha-m,$
      \item $f_2(x)=e^xL_m^{(\alpha)}(-x),\quad$ eigenvalue: $\alpha+1+m,$
       \item $f_3(x)=L_m^{(\alpha)}(x),\quad$ eigenvalue: $-m,$
       \item $f_4(x)=x^{-\alpha} e^x L_m^{(-\alpha)}(-x),\quad$ eigenvalue: $m+1,$
\end{enumerate}
where $L^{(\alpha)}_m(x)$ denotes the scalar Laguerre polynomial of degree $m$ and parameter $\alpha$, see e.g. \cite{AndrAR},\cite{Isma}, \cite{KoekLS}, \cite{KoekS},
\begin{equation}\label{eq:scalar-laguerre}
L_n^{(\alpha)}(x)=\frac{(\alpha+1)_n}{n!}\rFs{1}{1}{-n}{\alpha+1}{x}.
\end{equation}
We fix $m\in \N_0$ and take $f_1(x)$ as seed function. Theorem \ref{thm:ortogonalidad-excep} gives a sequence of exceptional Laguerre polynomials $(\widehat p_n^{(\alpha)})_n$. If we apply the Theorem \ref{thm:FourierAlgebrasDiagram} to the operator $x\in \mathcal{F}_R(P)$ and $\mathcal{L} = \psi^{-1}(x)$ we get
\begin{equation}
\label{eq:D-M-hat}
\widehat{P}_n\cdot BxA=(\Gamma_n-\lambda)\mathcal{L} \cdot\widehat{P}_n.
\end{equation} 
The operator $(-n-\alpha+m)\mathcal{L}$ is a three term recurrence relation
\begin{multline*}
(-n-\nu-1+m)\mathcal{L}=(-n-\alpha-1+m)\delta + (-n-\alpha+m)(2n+\alpha+1)\\+(-n-\alpha+1+m)n(n+\alpha)\delta^{-1}.
\end{multline*}
Let $(q_n)_n$ be a sequence of polynomials defined by this three term recurrence relation. After scaling and normalization, we get that these polynomials are
$$q_n(x^2)=\frac{(-1)^n}{(\alpha+1-m)_n}S_n(x^2+a^2;a,b,c),$$
where $S_n$ denotes the continuous dual Hahn polynomials and $c=\frac{\alpha}{2}+1,$ $a=\frac{\alpha}{2},$ $ b=\frac{\alpha}{2}-m.$
For $\alpha>m,$ the set $(q_n)_n$ forms a family of orthogonal polynomials, see \cite{KoekLS},\,\cite{KoekS}. If we change the seed function, we can proceed similarly and get continuous dual Hahn polynomials, but with a change in the parameters.

In \cite[(5-14)]{Koornwider85} it is shown that the Whittaker transform maps Laguerre polynomials to continuous dual Hahn polynomials. This observation states that there is a transform mapping exceptional Laguerre polynomials to continuous dual Hahn polynomials, and this can be derived directly from \cite[(5-14)]{Koornwider85}. 
\end{example}


\part{Matrix exceptional Laguerre polynomials}\label{part:matriXOL}

In the remaining part of the paper we apply the general theory of the previous sections 
to a special case, namely to the matrix exceptional Laguerre polynomials. These matrix exceptional polynomials arise from the matrix Laguerre polynomials
introduced and studied in \cite{KoelR}.


\section{Matrix Laguerre weight and differential operator}\label{sec:MLaguerreweightT0}
In this section we introduce the required input from the matrix Laguerre polynomials introduced in \cite{KoelR}. 
Let $N\geq1$ be a fixed integer, $(\mu_1,\dots,\mu_N)$ a sequence of non-zero coefficients, $\alpha,\nu>0$ and $(\delta_1^{(\nu)},\dots,\delta_N^{(\nu)})$ a sequence of positive numbers. Let us consider the following $N\times N$ Laguerre type weight matrix introduced on $(a,b)=(0,\infty)$ in \cite{KoelR}
\begin{equation}\label{eq:defWeightMLaguerrepols}
\begin{split}
&\qquad\qquad  W_{\mu}^{(\alpha,\nu)}(x)\,=\,L_{\mu}^{(\alpha)}(x)T^{(\nu)}(x)L_{\mu}^{(\alpha)}(x)^\ast, \\  
&T^{(\nu)}(x)_{i,j}= \de_{i,j} e^{-x}x^{\nu+i}\delta_{i}^{(\nu)}, 
\qquad L_{\mu}^{(\alpha)}(x)_{i,j}\,=\, 
\begin{cases}    \frac{\mu_i}{\mu_j}L_{i-j}^{(\alpha+j)}(x) &  i\geq j \\ 0 &  i<j\end{cases}
\end{split}
\end{equation}
where $L_n^{(\al)}$ denotes the scalar Laguerre polynomial of degree $n$ and parameter $\al$, see \eqref{eq:scalar-laguerre}. We consider the following second order symmetric differential operator 
\begin{equation}
\label{eq:opT0}
T_0=\frac{d^2}{dx^x}x+\frac{d}{dx}(M^{(\alpha,\nu)}_1x+M^{(\alpha,\nu)}_2)+C^{(\alpha,\nu)},
\end{equation}
with 
\begin{gather*}
M_1^{(\alpha,\nu)}=-(A_\mu+1)^{-1}, \quad M_2^{(\alpha,\nu)}=\nu+J+1+(\alpha+J)A_\mu, \\
C^{(\alpha,\nu)}=(\alpha-\nu)(A_\mu+1)^{-1}-J,
\end{gather*}
where $J_{i,j}=\de_{i,j} i,$ is a diagonal matrix and $(A_{\mu})_{i,j}=- \de_{j,i-1}\frac{\mu_i}{\mu_{i-1}}$ is a lower triangular matrix. Recall that we take $T_0$ as a second-order matrix differential operator acting from the right and  the 
coefficients of $T_0$ are polynomial functions. 
$T_0$ is symmetric with respect to the matrix 
weight $W_{\mu}^{(\alpha,\nu)}$, see \cite[Prop.~4.3]{KoelR}, when considered 
on the space of polynomials. In Proposition \ref{prop:T0essentiallyselfadjoint} we
show that $T_0$ is essentially self-adjoint, and we recall the corresponding monic
matrix orthogonal polynomials of \cite[Prop.~4.3]{KoelR} in 
\eqref{eq:MVorthogonalpolsW}, \eqref{eq:eigenvaluematrixLaguerrepols}.
In order to apply the results of Appendix \ref{sec:appA:MDEsingularities}, we need to take the adjoint of the action of $T_0$ on a row vector valued function $u$. So in order to solve the eigenvalue
equation $u(x)\cdot T_0 = \la \cdot u(x)$, we can apply the results of Appendix \ref{sec:appA:MDEsingularities} with $B_1(x) = x (M_1^{(\al,\nu)})^\ast + (M_2^{(\al,\nu)})^\ast$
and $B_2(x)=x((C^{(\al,\nu)})^\ast - \bar\la)$ which are polynomial of degree $1$ in $x$. In particular,
$B_1(0)=(M_2^{(\al,\nu)})^\ast$ is uppertriangular and $B_2(0)=0$.  So the indicial 
equation \eqref{eq:A-indicialeq} is 
\begin{equation}\label{eq:T0indicialeq}
\mu \mapsto \mu^N  \det(\mu- 1+ (M_2^{(\al,\nu)})^\ast) = \mu^N \prod_{j=1}^N (\mu+\nu+j)
\end{equation}
so the exponents are $\mu=0$ with multiplicity $N$ and $\mu=-\nu-j,$ $j\in\{1,\cdots,N\}$. 
The exponent $\mu=0$ and multiplicity $N$ corresponds to the  matrix 
orthogonal polynomials as eigenfunctions to $T_0$, and the other non-zero exponents reappear in Proposition \ref{prop:phim-seedfunction}. 

\begin{rmk}
\label{rmk:Laguerre-diag}
    The structure \eqref{eq:defWeightMLaguerrepols} of $W^{(\nu)}$ coincides with the condition \eqref{eq:weightstructure5} of Section \ref{sec:T0-diagonalizable}. On the other hand, the differential operator $T_0$ can be conjugated into a diagonal differential operator, see \cite[Prop. 4.3]{KoelR}. Now, the construction in Section 5 leads to a family of matrix valued exceptional Laguerre polynomials with singular leading coefficients. In Part 2, we follow Sections \ref{sec:2ndoperatorsfacDarbux}, \ref{sec:symmdiffopexceptweights} and \ref{sec:matriXOL} for a particular seed function, and we obtain families of matrix valued exceptional Laguerre polynomials with invertible leading coefficients.
\end{rmk}


\section{Seed functions for $T_0$}\label{sec:seedmatrixeigenfunctionsforT0}
According to Section \ref{ssec:scalar-eigenvalue}, a suitable seed function $\phi$ is a matrix eigenfunction for the operator $T_0$ with scalar eigenvalue $\la$, see \eqref{eq:seedfunctiongeneral}. In analogy with the scalar Laguerre case, the differential operator $T_0$ has four types of eigenfunctions with
scalar eigenvalue having a polynomial part:
\begin{enumerate}
    \item $G_1(x) = F(x)$,
    \item $G_2(x) = x^{-\nu-J}F(x)$, 
    \item $G_3(x) = e^x F(x)$,
    \item $G_4(x) = e^x x^{-\nu-J} F(x)$,
\end{enumerate}
where the $F(x)'s$ are specific lower triangular polynomials in each of the cases. The particular form of the seed function is indicated by the solutions of the indicial equation \eqref{eq:T0indicialeq}. In this paper we will only treat exceptional polynomials related with case (2). Case (3) can be treated in a similar way and cases (1) and (4) do not lead to exceptional weight matrices with finite moments.

In order to state the result, we introduce  a sequence of lower triangular matrix 
polynomials 
$(F_m(x))_{m\in\N_0}$ defined explicitly in terms of hypergeometric functions by
\begin{equation}\label{eq:functions-Fm}
(F_m(x))_{i,j}= \frac{(-1)^{i-j-1}(-i)_j(-\alpha-i)_{i-j}\mu_1}{i(\alpha+2)_{i-1}\mu_j}\rFs{2}{2}{-m,-\alpha-j}{1-\nu-i,-\alpha-i}{x}
\end{equation}
and we assume that the ${}_2F_2$-series is well-defined. Since we assume $\al,\nu>0$ it suffices to assume that neither $\al$ nor $\nu$ is an integer less than $m$, and we assume this for the rest of the paper.

\begin{rmk}\label{rmk:functions-Fm} 
(i) Since $(-i)_j=0$ for $j>i$, we see that $F_m$ is a lower triangular matrix polynomial of degree $m$. Indeed, each of its non-zero entries is a polynomial of degree $m$. Hence, the leading coefficient of $F_m(x)$ is a full lower triangular matrix.
\smallskip

\noindent (ii) The diagonal elements of $F_m$ are multiples of scalar Laguerre polynomials and more generally we have 
for $i\geq j$
\begin{align*}
(F_m(x))_{i,j}= 
\frac{(-1)^{i-j-1}(-i)_j(-\alpha-i)_{i-j}\mu_1}{i(\alpha+2)_{i-1}\mu_j}
\frac{m!}{(1-\nu-i)_m}
\sum_{p=0}^{m\wedge i-j} \frac{(j-i)_p \, x^p}{p!\, (-\al-i)_p}
L^{(p-\nu-i)}_{m-p}(x).
\end{align*}
Here we use the general expansion \cite[(15), p.~439]{PrudBM}, except that $z^k$ is missing in the summand on the right hand side of \cite[(15), p.~439]{PrudBM}. 
The proof of the identity \cite[(15), p.~439]{PrudBM} follows by using 
the Chu-Vandermonde sum $\frac{(-\al-j)_k}{(-\al-i)_k} = {}_2F_1(-k, j-i; -\al-i;1)$ 
in the expansion \eqref{eq:functions-Fm}, interchanging summations and writing the 
terminating ${}_1F_1$ in the summand as a Laguerre polynomial.

\end{rmk}

\begin{prop}\label{prop:phim-seedfunction}
The functions $\phi_m(x)=x^{-\nu-J}F_m(x)$ are seed functions for $T_0$ with scalar eigenvalue $(\alpha-m)$ for all $m\in\N_0$. 
\end{prop}

The row vectors of $\phi_m$ then correspond to solutions for the non-zero exponents for the indicial equation \eqref{eq:T0indicialeq}. Note that we can multiply by a constant diagonal matrix on the left. Note that no row $\phi$ is contained in 
$L^2(W_{\mu}^{(\alpha,\nu)})$.
\begin{proof} 
Since the powers in $x^{-\nu-J}$ correspond to the solutions of the indicial equation 
\eqref{eq:T0indicialeq}, we know that $x^{-\nu-J}F(x)$ with $F$ analytic in 
a neighbourhood of the origin is a solution to $\phi\, T_0=\la\, \phi$. Conjugating  $T_0$ with $x^{-\nu-J}$ from the left, we see that $F$ has to satisfy the mixed relation
\begin{equation}\label{eq:seedfunctiondDEF}
\begin{split}
&F''(x) x + F'(x)(xM_1^{(\al,\nu)}+M_2^{(\al,\nu)}) -2(\nu+J)F'(x) 
+ (\nu+J)(\nu+J+1)\frac{1}{x} F(x) \\ 
&-(\nu+J) \frac{1}{x} F(x)(xM_1^{(\al,\nu)}+M_2^{(\al,\nu)}) +F(x) (C^{(\al,\nu)}-\la) 
= 0.
\end{split}
\end{equation}
Since $F$ has to be analytic, the singularity related to $\frac{1}{x}$ in 
\eqref{eq:seedfunctiondDEF} is apparent. So we have $(\nu+J+1)F(0)=F(0)M_2^{(\al,\nu)}$. 
Moreover, since all the matrices acting from the right on $F$ in \eqref{eq:seedfunctiondDEF} are lower 
triangular, we have a lower triangular solution $F$. Then for $g=F_{i,i}$ we get the 
well-known differential equation
\begin{equation*}
xg''(x) + (-x+1-\nu-i)g'(x) + (\al-\la)g(x)=0,
\end{equation*}
which is the Laguerre differential equation. Since we want a polynomial solution, we 
choose $\la = \al-m$, and then $F_{i,i}$ is a constant times the Laguerre 
polynomial $L_m^{(-\nu-i)}$ and we choose the multiple as in Remark 
\ref{rmk:functions-Fm}. 
Now for the analytic functions $F_{i,j}(x)$, $i>j$, we get an inhomogeneous differential equation involving $F_{i,k}(x)$, $j<k\leq i$. The analytic solution is completely determined by $F_{i,j}(0)$, which in turn follows from $F_{i,j+1}(0)$ since
$(\nu+J+1)F(0)=F(0)M_2^{(\al,\nu)}$. We put 
$F_{i,j}(x) = \sum_{k=0}^\infty c_{i,j,k}x^k$. Then for $i>j$ we have 
\[
(\nu+i+1) c_{i,j,0} = (\nu+j+1) c_{i,j,0} - (\al+j+1)\frac{\mu_{j+1}}{\mu_j} c_{i,j+1,0}
\]
and $c_{i,j,0} = - \frac{\al+j+1}{i-j}\frac{\mu_{j+1}}{\mu_j}c_{i,j+1,0}$ is indeed
satisfied by \eqref{eq:functions-Fm}. 

It remains to show that with $F(x)=\sum_{k=0}^\infty F_k x^k$, 
\begin{equation}\label{eq:valueFk-cijk}
F_k= (c_{i,j,k})_{i,j=1}^N, \qquad c_{i,j,k}=\frac{(-1)^{i-j-1}(-i)_j(-\alpha-i)_{i-j}\mu_1}{i(\alpha+2)_{i-1}\mu_j} \frac{(-m)_k(-\alpha-j)_k}{(1-\nu-i)_k(-\alpha-i)_k k!}
\end{equation}
we obtain a solution to \eqref{eq:seedfunctiondDEF} for $\la =\al-m$. In order to have simpler recursions, we multiply \eqref{eq:seedfunctiondDEF} on the right with $(1+A_\mu)$ to cancel the inverses $(1+A_\mu)^{-1}$ in $M_1^{(\al,\nu)}$ and $C^{(\al,\nu)}$. 
Plugging the power series expression for $F$ in the resulting matrix differential
equations gives a power series identity. We have already verified that the coefficient of $x^{-1}$ is indeed equal to $0$. The coefficient of $x^k$, $k\in \N_0$, equals, after a straightforward calculation, 
\begin{align}
\label{eq:recursionFks}
&k(k+1) F_{k+1}(1+A) + (k+1)F_{k+1} M_3 - (k+1)(\nu+J)F_{k+1}(1+A) \\
&\ + (\nu+J)(\nu+J+1)F_{k+1}(1+A) - (\nu+J)F_{k+1} M_3 -kF_k + (\nu+J)F_k 
+ F_k M_4 \nonumber
\end{align}
where $A= A_\mu$, $M_3=(\nu+J+1)(1+A) + (\al+J)A(1+A)$, and $M_4=(m-\nu) + 
(m-\al)A - J(1+A)$. Note that the matrices acting on the right in \eqref{eq:recursionFks} are lower triangular and band limited, and the ones acting on the left are diagonal. 
Calculating the $(i,j)$-th entry of \eqref{eq:recursionFks} gives terms
of the form $c_{i,j+r,k+s}$ for $r\in\{0,1,2\}$, $s\in \{0,1\}$. An explicit
calculation gives that $(i,j)$-th entry of \eqref{eq:recursionFks} equals
\begin{align*}
&c_{i,j,k+1}\bigl( (k+1)(k-\nu+1+j-2i)+(\nu+i)(i-j)\bigr) \\
+ &c_{i,j+1,k+1} \frac{-\mu_{j+1}}{\mu_j} \bigl( (k+1)(k-\nu+1+2(j-i)+\al) + (\nu+i)(i-2j-\al)\bigr) \\ +
&c_{i,j+1,k+1} \frac{\mu_{j+2}}{\mu_j} (\al+j)(k+1-\nu-i) 
+ c_{i,j,k}(m-k+i-j) + c_{i,j+1,k} \frac{-\mu_{j+1}}{\mu_j}(m-\al-j).
\end{align*}
It is now a straightforward calculation that for the value of $c_{i,j,k}$ as  
in \eqref{eq:valueFk-cijk} this yields zero. 
\end{proof}

Note that the indicial equation \eqref{eq:T0indicialeq} 
predicts the form of the solution, but this does not explain that the eigenvalue
is indeed constant and that the remaining analytic term $F$ is actually polynomial. 

\begin{rmk}
As it has been pointed out in Remark \ref{rmk:Laguerre-diag}, the differential operator $T_0$ can be conjugated into a diagonal differential operator $T_0^d$, see \cite[Prop. 4.3]{KoelR}. The seed function $\phi_m$ on Proposition \ref{prop:phim-seedfunction} can be decomposed as:
$$\phi_m(x)=M\phi^d_m(x)L_{\mu}^{(\alpha)}(x)^{-1},$$
where $\phi_m^{d}(x)=\text{diag}(x^{-\nu-i}L_{m}^{(-\nu-i)}(x))$ is an eigenfunction of the diagonal operator $T_0^d$ and $M$ is a constant lower triangular matrix.
\end{rmk}


\section{Intertwiner relations for the matrix Laguerre operator}\label{sec:intertwingLaguerre}
In this section, we introduce an intertwinning operator $A_m$ for each seed function $\phi_m$. We show that $A_m$ preserves polynomials and the regularity of the leading coefficients. 

By Proposition  \ref{prop:phim-seedfunction} the matrix function $\phi_m(x)=x^{-\nu-J}F_m(x)$ is a seed function of the second order differential operator $T_0$ for the scalar eigenvalue $(\alpha-m)$. Recall that $F_m$ is a lower triangular matrix 
polynomial of degree $m$, so that 
\begin{equation}\label{eq:detFm}
\begin{split}
\det(F_m(x))&\,=\, \prod_{k=1}^N \frac{(-1)^{k+1}(k-1)!m!\mu_1}{(\alpha+2)_{k-1}(1-\nu-k)_m\mu_k}L_m^{(-\nu-k)}(x) \\
&\, =\, \prod_{k=1}^N \frac{(-1)^{k+1}(k-1)!\, \mu_1}{(\alpha+2)_{k-1}\, \mu_k}
\rFs{1}{1}{-m}{1-\nu-k}{x}
\end{split}
\end{equation}
and $\det(F_m)$ is a polynomial of degree $mN$.  
For $\nu>\max(0,m-1)$ we see that each of the ${}_1F_1$-series in 
\eqref{eq:detFm} has positive coefficients for each power of $x$, so that it follows 
that the zeros of $\det(F_m)$ are not contained in $[0,\infty)$. 

The seed function $\phi_m$ has a singularity at $x=0$, and 
is single valued on the cut complex plane $\C\setminus (-\infty,0]$. 
Note that $\phi_m$ is invertible on $\C\setminus (-\infty,0]$ minus the finite set of zeros of $\det(F_m)$. Then we have 
\begin{equation}\label{eq:phiminverseddxphim}
\phi_m(x)^{-1}\phi_m'(x)=-\frac{\nu}{x}-\frac{1}{x}F_m(x)^{-1}JF_m(x)+F_m(x)^{-1}F_m'(x).
\end{equation}
Note that this is a rational matrix function and that its poles are not contained in 
$(0,\infty)$ for $\nu>\max(0,m-1)$. 
We use the matrix polynomial $\Up_m(x)=x\det(F_m(x))\Id$ to cancel the 
singularities in \eqref{eq:phiminverseddxphim}.

\begin{lem}\label{lem:degreephiinvderphiUps}
$x\mapsto \phi_m(x)^{-1}\phi_m'(x)\Up_m(x)$ is a matrix polynomial of degree $mN$ with 
invertible leading coefficient. 
\end{lem}

\begin{proof} Recall that each entry of $F_m$ is a polynomial of degree $m$. It follows
that each entry of the classical adjoint, or adjugate, matrix $\adj(F_m)$ is a polynomial of degree $(N-1)m$, since it is the determinant of $F_m$ with one row and one column removed. By Cramer's rule $\det(F_m) F_m^{-1}= \adj(F_m)$. So 
\eqref{eq:phiminverseddxphim} gives 
\begin{equation*}
\phi_m(x)^{-1}\phi_m'(x) \Up_m(x) =- \nu\det(F_m(x))-\adj(F_m(x))JF_m(x)+
x\, \adj(F_m(x))F_m'(x)
\end{equation*}
and each of the terms on the right hand side is a matrix polynomial of degree $mN$.
Moreover, it is a lower triangular matrix polynomial. 
So we can calculate the diagonal entries of $\phi_m(x)^{-1}\phi_m'(x) \Up_m(x)$, and 
for $1\leq i \leq N$ we find 
\begin{equation*}
(\phi_m(x)^{-1}\phi_m'(x) \Up_m(x))_{i,i} = 
\bigl( (-\nu-i)F_m(x)_{i,i} - x F_m'(x)_{i,i}\bigl) \prod_{p\not =i} F_m(x)_{p,p},
\end{equation*}
which is a non-zero polynomial of precise degree $mN$. It follows that the 
leading coefficient of $\phi_m(x)^{-1}\phi_m'(x) \Up_m(x)$ is an invertible
lower triangular matrix. 
\end{proof}

Now we can apply the results of Part \ref{part:general}, and 
we define the first order matrix differential operator $A_m$ acting from the right by 
\begin{equation}\label{eq:intertwinerAm}
A_m=\frac{d}{dx}\Upsilon_m-\phi_m^{-1}\phi_m'\Upsilon_m.
\end{equation}
Then $A_m$ maps row, respectively matrix, polynomials to 
row, respectively matrix, polynomials increasing the degree by $mN$. This follows  
by Lemma \ref{lem:degreephiinvderphiUps} and noting  that  $\Up_m$ is a polynomial of degree $mN+1$. 
\begin{lem}
\label{lem:Am-preserves-regularity}
The operator $A_m$ preserves regularity of leading coefficients of polynomials.
\end{lem}
\begin{proof}
It is enough to check that $x^n\cdot A_m$ has invertible leading coefficient for all $n\in\mathbb{N}_0.$ By definition of $A_m$, $x^n\cdot A_m=x^{n-1}n\Upsilon_m(x)-x^n\phi_m(x)^{-1}\phi_m'(x)\Up_m(x)$
which is a lower triangular polynomial of degree less than or equal to $mN+n$. We can calculate its diagonal entries, and for $1\leq i \leq N$ we find 
\begin{align*}
    (x^n\cdot A_m)_{i,i}&= x^{n-1}n\Upsilon_m(x)-x^n \bigl( (-\nu-i)F_m(x)_{i,i} - x F_m'(x)_{i,i}\bigl) \prod_{p\not =i} F_m(x)_{p,p} \\
     &= \bigl( nF_m(x)_{i,i}+(\nu+i)F_m(x)_{i,i} + x F_m'(x)_{i,i}\bigl) x^n\prod_{p\not =i} F_m(x)_{p,p}.
\end{align*}
By looking at the leading coefficient, we verify that $nF_m(x)_{i,i}+(\nu+i)F_m(x)_{i,i} + x F_m'(x)_{i,i}$ is a polynomial of degree $m$. It follows that $x^n\cdot A_m$ has precise degree $mN+n$ and that its leading coefficient is an invertible lower triangular matrix.
\end{proof}

Following \eqref{eq:Bhat} we define the first order matrix
differential operator $B_m$ acting from the right by 
\begin{align}\label{eq:intertwiner-Bn}
\begin{split}
 B_m&=\Up_m(x)^{-1}\left(\frac{d}{dx}x+M_1^{(\al,\nu)}x+M_2^{(\al,\nu)}+\phi_m(x)^{-1}\phi_m'(x)x\right)\\
&=\frac{d}{dx}\Up_m^{-1}(x)x+(\Up_m^{-1}(x))'x+\Up_m^{-1}(x)(M^{(\al,\nu)}_1x+M^{(\al,\nu)}_2)+\Up_m^{-1}(x)\phi_m^{-1}(x)\phi'_m(x)x.
\end{split}
\end{align}
and Lemma \ref{lem:TO=AB+la} holds.


\section{Matrix exceptional Laguerre polynomials}\label{sec:MatrixXOLaguerre}
In this section we introduce the exceptional weight  and the sequence of matrix exceptional Laguerre 
polynomials. We assume $\nu>\max(0,m-1)$, so that the zeros 
of $\det(F_m)$ are not contained in $[0, \infty)$. Since there are finitely many zeros, we have 
\begin{equation}\label{eq:distancezeroesdetFmto0infty}
\de = \min \{ d(z) \mid z\in \C, \det(F_m(z))=0 \} >0 
\end{equation}
where $d(z) = \inf\{ |x-z| \mid x\in [0,\infty)\}$ is the distance of $z\in \C$ to $[0,\infty)$. 
Following Definition \ref{defn:exceptional-weight} we define the matrix weight function on $(0,\infty)$ by 
\begin{equation}\label{eq:defnweight-excep-laguerre}
\widehat{W}_\mu^{(\alpha,\nu,m)}(x)=\frac{W_\mu^{(\alpha,\nu)}(x)}{x\left(\det(F_m(x))\right)^{2}}.
\end{equation}
Because we can absorb the factor $x$ in the numerator into the diagonal part $T^{(\nu)}$ of \eqref{eq:defWeightMLaguerrepols} we
see that $\widehat{W}_\mu^{(\alpha,\nu,m)}$ is a positive definite matrix on $[0,\infty)$ and $\det(\widehat{W}_\mu^{(\alpha,\nu,m)}(x))>0$ for $x>0$. Moreover, this remark with 
\eqref{eq:distancezeroesdetFmto0infty} implies that all moments 
of the matrix weigh $\widehat{W}_\mu^{(\alpha,\nu,m)}$ on $(0,\infty)$ exist. 

We note that the exceptional weight $\widehat{W}_\mu^{(\alpha,\nu,m)}$ is reducible to weights of smaller size if and only if the Laguerre weight $W_\mu^{(\alpha,\nu)}(x)$ is reducible, since they only differ by a scalar factor. Extensive computations indicate that the Laguerre weight is irreduceble but we do not have a proof, see the discussion in \cite[\S 1]{KoelR}.

\begin{prop}\label{prop:polsdenseinWandhatW}
The space $\Pol$ of row vector valued polynomials is dense in 
$L^2(W^{(\al,\nu)}_\mu)$ and dense in $L^2(\widehat{W}^{(\al,\nu,m)}_\mu)$.
\end{prop}

\begin{proof}
Note first that $\Pol$ is contained in the $L^2$-spaces, since both have finite moments. 
For the measure $\widehat{W}^{(\al,\nu,m)}_\mu$ we see that the trace measure $\tau$ is 
absolutely continuous with respect to Lebesgue measure $dx$ on $(0,\infty)$ and its Radon-Nikodym derivative is 
\begin{equation}\label{eq:RNdertraceLebesgue}
\frac{d\tau}{dx} = \frac{1}{(\det(F_m(x)))^2} e^{-x} \sum_{i=1}^N \de^{(\nu)}_i x^{\nu+i-1} 
\end{equation}
by \eqref{eq:defWeightMLaguerrepols}. 
Because of \eqref{eq:distancezeroesdetFmto0infty} we see that the function 
$\exp(\be |x|)$ is integrable with respect to the trace measure for some positive $\be$, 
e.g. $\be=\frac{1}{2}$. By \cite[Thm.~5.2, p.~80]{Freu}
the trace measure $\tau$ is determinate, and hence polynomials are dense in the 
Hilbert space $L^2(\tau)$, the weighted $L^2$-space with respect to the 
trace measure as Borel measure on $(0,\infty)$, see \cite{Akhi}.  

We write the matrix measure as $V\tau= (V_{i,j})_{i,j=1}^N\tau$. 
Assume that $\xi \in L^2(\widehat{W}^{(\al,\nu,m)}_\mu)$, $\xi=(\xi_1,\cdots, \xi_N)$,  is orthogonal to the space $\Pol$. 
By taking a row vector polynomial which is zero except in the $j$-th coefficient, 
$(0,\cdots,0,p,0,\cdots,0)$, we see that 
\[
\int_0^\infty \Bigl( \sum_{i=1}^N \xi_i(x) V_{i,j}(x)\Bigr) \overline{p(x)} \, d\tau(x) = 0
\]
for all scalar polynomials $p$. Since $\tau$ corresponds to a (scalar) determinate moment problem, it follows that $\sum_{i=1}^N \xi_i V_{i,j}=0$ in $L^2(\tau)$. Since $j$ is arbitrary, we find $\xi V = (0,\cdots, 0)$ as row vector valued function $\tau$-a.e. 
Since $\det(V(x)) (\frac{d\tau}{dx})^N = 
\det(\widehat{W}^{(\al,\nu,m)}_\mu(x))\not=0$ for $x\not=0$, with the notation as in  \eqref{eq:RNdertraceLebesgue},
we find that $\xi=0$ in $L^2(\widehat{W}^{(\al,\nu,m)}_\mu)$. So 
$\Pol$ is dense in $L^2(\widehat{W}^{(\al,\nu,m)}_\mu)$.
The proof for $L^2(W^{(\al,\nu)}_\mu)$ is analogous. 
\end{proof}

Before continuing we sharpen the result of \cite[Prop.~4.3]{KoelR}.
We recall that we have the sequence of monic matrix orthogonal polynomials 
$(P_n^{(\alpha,\nu)})_n$ with respect to $W_\mu^{(\alpha,\nu)}$, i.e. 
$P_n^{(\alpha,\nu)}$ is a matrix polynomial with leading coefficient the identity matrix
and satisfying 
\begin{equation}\label{eq:MVorthogonalpolsW}
\int_0^\infty P_n^{(\al,\nu)}(x)\, W_\mu^{(\al,\nu)}(x) 
\bigl(P_m^{(\al,\nu)}(x)\bigr)^\ast\, dx = \de_{m,n} H^{(\al,\nu)}_n,
\end{equation}
where the matrix integration is done entrywise. Here $H_n$ is a constant positive definite matrix, which is the matrix squared norm of $P_n^{(\alpha,\nu)}$. 
Then \cite[Prop.~4.3]{KoelR} states that as a matrix differential operator acting on the right on a matrix function we have
\begin{equation}\label{eq:eigenvaluematrixLaguerrepols}
P_n^{(\al,\nu)}\cdot T_0=\Ga_n^{(\alpha,\nu)}\cdot P_n^{(\al,\nu)}, \qquad  \Ga^{(\al,\nu)}_n=(-n+\al-\nu)(A_\mu+1)^{-1}-J.
\end{equation}
Note that $\si(\Ga_n^{(\alpha,\nu)})=\{ -n+\al-\nu-j\mid j\in \{1,\cdots, N\}\}$ and 
diagonalising gives 
$R_n\Ga_n^{(\alpha,\nu)} = D_n^{(\al,\nu)} R_n$, with 
$D_n^{(\al,\nu)}$ the diagonal matrix with $-n+\al-\nu-j$ at its $(j,j)$-th entry. Note
that $R_n$ is lower triangular and invertible. 
Then the $j$-the row of $R_nP^{(\al,\nu)}_n$ is a polynomial of degree $n$ in $x$, hence an 
element of $\Pol$, and an
eigenfunction of $T_0$ for the eigenvalue $-n+\al-\nu-j$. We denote this element by $p_n^j\in \Pol$.
Then the set eigenvalues of $T_0$ is $\al-\nu-1-\N_0$, and the multiplicity of the 
eigenvalue $\al-\nu-1-p$, $p\in \N_0$, is $\min(N,p+1)$ and an orthogonal basis of 
the eigenspace for the eigenvalue  $\al-\nu-1-p$ is given by $\{ p_n^j\mid n+j=p+1\}$. 
Indeed, since
all the degrees are different orthogonality follows from \eqref{eq:MVorthogonalpolsW} after multiplying from the left by $R_n$ and the right from $R_n^\ast$. Since the eigenspaces in $\Pol$ of $T_0$ 
are orthogonal for different eigenvalues, we see that $(p_n^j)_{n,j}$ gives an orthogonal basis of eigenvectors.

\begin{prop}\label{prop:T0essentiallyselfadjoint}
The operator $T_0$ with $D(T_0)=\Pol$ is essentially self-adjoint on 
$L^2(W^{(\al,\nu)}_\mu)$. Its closure $\overline{T_0}$ has compact 
resolvent and $\si(\overline{T_0}) =\al-\nu-1-\N_0$. 
\end{prop}
\begin{proof} It follows from \cite[\S 4]{KoelR} that $T_0$ with $D(T_0)=\Pol\subset
L^2(W^{(\al,\nu)}_\mu)$
is symmetric. Moreover, from \eqref{eq:opT0} we see that $T_0$ preserves 
the space of row polynomials and even its degree. By Proposition \ref{prop:polsdenseinWandhatW}, $T_0$ is densely defined and we check that $D(T_0^\ast)$ 
is the maximal domain, i.e. 
\begin{equation*}
D(T_0^\ast) = \{ v = \sum_{n,j} c_n^j p_n^j \in L^2(W^{(\al,\nu)}_\mu) \mid
\sum_{n,j} c_n^j (\al-\nu-n-j) p_n^j \in L^2(W^{(\al,\nu)}_\mu) \}
\end{equation*}
with $T_0^\ast \sum_{n,j} c_n^j p_n^j = \sum_{n,j} c_n^j (\al-\nu-n-j) p_n^j$. 
Since any element of the graph of $T^\ast_0$ is also in the closure of the graph of 
$T_0$ by approximating by finite sums, we see that $T_0^\ast=\overline{T_0}$ and 
$T_0$ is essentially self-adjoint. 
The statement on the compact resolvent follows, since the eigenspaces are finite-dimensional and the eigenvalues diverge to $-\infty$. 
\end{proof}
\begin{rmk}
The condition $\nu>\max(0,m-1)$ implies that $\si(\overline{T_0})< \al-m$. 
\end{rmk}

\begin{lem}\label{lem:AmBmdenselydefinedoperators} 
The first order matrix differential operators $A_m$ and $B_m$ are 
densely defined operators with 
\begin{equation*}
\begin{split}
&A_m \colon D(A_m)=\Pol \subset L^2(W^{(\al,\nu)}_\mu) 
\to L^2(\widehat{W}^{(\al,\nu,m)}_\mu), \\
&B_m \colon D(B_m)=\Pol \subset  L^2(\widehat{W}^{(\al,\nu,m)}_\mu)
\to L^2(W^{(\al,\nu)}_\mu). 
\end{split}
\end{equation*}
\end{lem}

\begin{proof} Since $\Pol$ is dense in both $L^2(W^{(\al,\nu)}_\mu)$ and 
$L^2(\widehat{W}^{(\al,\nu,m)}_\mu)$ by 
Proposition \ref{prop:polsdenseinWandhatW} we see that both operators have dense domain.
Since $A_m$ preserves the space $\Pol$, cf. \eqref{eq:defA}, 
we see that $A_m$ maps $\Pol$ into $L^2(\widehat{W}^{(\al,\nu,m)}_\mu)$. 
So $A_m$ is well-defined. 

It remains to show that $B_m$ maps 
$\Pol$ into $L^2(W^{(\al,\nu)}_\mu)$, and we see that by 
\eqref{eq:intertwiner-Bn} the operator $B_m$ does not preserve polynomials
in general. Since $\Up_m(x) = x\det(F_m(x))$, and the zeros of $\det(F_m(x))$ are 
off $[0,\infty)$ by the assumption $\nu>\max(0,m-1)$, cf. 
\eqref{eq:distancezeroesdetFmto0infty}, 
so that $0<|\det(F_m(x))|^{-1}<M$ on $[0,\infty)$. Considering 
\eqref{eq:intertwiner-Bn} and using \eqref{eq:phiminverseddxphim} we see that 
$B_m$ maps $\Pol$ into $L^2(W^{(\al,\nu)}_\mu)$ if $\frac{1}{x} \in L^2(W^{(\al,\nu)}_\mu)$. Because of \eqref{eq:defWeightMLaguerrepols} we see that the diagonal
elements in the diagonal matrix $x^{-2}T^{(\nu)}$ are $\de^{(\nu)}x^{\nu+i-2}e^{-x}$ 
for $i\in \{1,\cdots, N\}$, so that all terms are integrable because $\nu>\max(0,m-1)$. 
Hence, $B_m$ maps $\Pol$ into $L^2(W^{(\al,\nu)}_\mu)$.
\end{proof}

In order to have the results of Section \ref{sec:symmdiffopexceptweights} and 
Section \ref{sec:matriXOL} we need to check the conditions in 
Lemma \ref{lem:strange-relation} and Proposition \ref{prop:-BadjA}. 
Recall that $F_m$ has no zeros on $[0,\infty)$ and that $|\det(F_m)|$ is bounded from zero and from above. For Lemma \ref{lem:strange-relation} we see that both terms vanish for the 
limit to $\infty$ using the fact the $e^{-x}$ in the weight kills all powers of $x$. 
For Proposition \ref{prop:-BadjA} we see in the same way that the boundary term
vanishes at $\infty$ for any two polynomials $p,q\in \Pol$. Moreover, at $0$ the term 
also vanishes since the powers of $x$ are strictly positive. 

Note that $\al-m$ is not in the spectrum $T_0^\ast$, so that the spectrum of the self-adjoint operators $\overline{T_0}$ and $S_1$ is the same up to possibly one point. Since the self-adjoint operator $\overline{T_0}$ has compact resolvent and a basis of eigenfunctions in $\Pol$ 
by Proposition \ref{prop:T0essentiallyselfadjoint}, and since $A$ preserves $\Pol$, we see that $p_n^j\cdot A_m$ are orthogonal eigenfunctions for $S_1$ for the same eigenvalue $-\al-\nu-n-j$.

\begin{prop}
\label{prop:ker-B}
Let $f=(f_1,\ldots,f_N)\in\Ker(B_m)$ then $f_i(x)=e^x r_i(x),$ where $r_i$ is a rational function of $x.$ Hence in $L^2(\widehat{W}_\mu^{(\alpha,\nu,m)})$ the kernel of $B_m$ is trivial. 
\end{prop}


\begin{proof}
Let $f\in \Ker(B_m),$ then $g(x)=f(x)\Upsilon(x)^{-1}e^{-x} $ satisfies
\begin{equation}
    \label{eq:B-exp}
g\cdot \left(\frac{d}{dx}x+(M_1^{(\alpha,\beta)}+1)x+M_2^{(\alpha,\beta)}+\phi_m(x)^{-1}\phi'_m(x)x \right)=0.    
\end{equation}
 The proof follows by proving that $g_i$ is a rational function for all $i.$ For, the $(1,N)$ entry of \eqref{eq:B-exp} is
$$(g_N(x)xF_m(x)_{N,N})'=0 \quad\text {then}\quad g_N(x)=\frac{1}{xF_m(x)_{N,N}}.$$
More generally and proceding recursively, the $(1,j)$ of \eqref{eq:B-exp} is
$$(g_j(x)xF_m(x)_{j,j})'=R_j(x),$$
where $R_j$ is a rational function involving $g_{N}, \ldots,g_{j+1},$ and $F_m(x)_{N,N}, \ldots,F_m(x)_{{j+1},{j+1}}$ and then $g_j$ is a rational function of $x.$
\end{proof}

\subsection{Matrix valued exceptional Laguerre polynomials}
In this subsection we introduce a sequence of matrix exceptional Laguerre polynomials. We recall that we have the sequence of monic matrix orthogonal polynomials $(P^{(\alpha,\nu)}_n)_n$ with respect to $W_\mu^{(\alpha,\nu)},$ see \eqref{eq:MVorthogonalpolsW}. We fix $m\in\mathbb{N}_0$ and following Theorem \ref{thm:ortogonalidad-excep} we denote 
\begin{equation}
\label{eq:X-laguerre}
\widehat{P}^{(\alpha,\nu,m)}_n=P_n^{(\alpha,\nu)}\cdot A_m.
\end{equation}
These are the matrix valued exceptional Laguerre polynomials associated to $(T_0, \phi_m, \Upsilon_m)$.

\begin{thm}
Assume $\nu>\max(0,m-1),$ then $\widehat{P}_n^{(\alpha,\nu,m)}$ is a matrix polynomial of degree $mN+n$ with invertible leading coefficient whose rows form an orthogonal basis for $L^2(\widehat{W}_\mu^{(\alpha,\nu,m)})$. Moreover, $\widehat{P}^{(\alpha,\nu,m)}_n$ is an eigenfunction of $T_1$ of eigenvalue $\Gamma_n^{(\alpha,\nu)}$.
\end{thm}
\begin{proof}
The polynomial $\widehat{P}^{(\alpha,\nu,m)}_n$ is a matrix valued polynomial of degree $mN+n.$ Moreover, its leading coefficient is lower triangular and invertible, see Lemma \ref{lem:Am-preserves-regularity} and its proof. By Theorem \ref{thm:ortogonalidad-excep}, we get that $\widehat{P}^{(\alpha,\nu,m)}_n$ is an eigenfunction of $T_1$ of eigenvalue $\Gamma_n^{(\alpha,\nu)}$ and that the exceptional polynomials are orthogonal with respect to $\widehat{W}_\mu^{(\alpha,\nu,m)}$. Finally, Proposition \ref{prop:ker-B} implies $\alpha-m \notin \sigma_p(S_1)$ and by Proposition \ref{prop:densityofexceptpols} the rows of $\widehat{P}_n^{(\alpha,\nu,m)}$ are dense in $L^2(\widehat{W}_\mu^{(\alpha,\nu,m)})$.
\end{proof}

\begin{rmk}
In the scalar case, i.e. $N=1$, exceptional Laguerre polynomials have been studied by several authors, see for instance \cite{UllaM}, \cite{ArnoN}, \cite{Liaw}, \cite{MarcellanU}, \cite{Sasaki}, \cite{Duran1}, \cite{UKM1}, \cite{Sasa1}, \cite{Sasa2}, \cite{Sasa4}. We end this section by linking our notation of exceptional Laguerre polynomials $\widehat{P}_n^{(\alpha,\nu,m)}$ with some notations in the literature.

In \cite{Liaw}, \cite{UllaM}, \cite{MarcellanU}, the exceptional Laguerre polynomials are classified as type I, type II, and type III $X_m$-Laguerre polynomials. The exceptional Laguerre polynomials $\widehat{P}_n^{(\alpha,\nu,m)}$ correspond to the type II of $X_m$-Laguerre polynomials \cite[Section 4.2]{UllaM}: 
$$\widehat{P}_{n-m}^{(\alpha,\nu+1,m)}=-L^{II,\nu}_{n,m}=A^{(\nu+1)}_m\cdot L_{n-m}^{(\nu+1)},$$
where $L_{n-m}^{(\nu+1)}$ denotes the scalar Laguerre polynomial of degree $n-m$ and parameter $\nu+1.$ Note that the action of the intertwining operator $A^{(\nu+1)}_m$ is from the left as standard notation in the scalar case \cite[Eq. (88)]{UllaM}. Moreover in \cite{UllaM} the authors evaluate the interwiner operator $A^{(\nu+1)}_m$  in $L_{n-m}^{\nu+1}$ so that $L^{II,\nu}_{n,m}$ becomes a polynomial of degree $n.$ 

In terms of partitions, in \cite[Prop. 4, Eq. (69)]{ArnoN}, the authors show that type II of $X_m$-Laguerre polynomials corresponds to
$$L^{II,\nu}_{n,m}=c_{n,m} L_{\emptyset, (1,\ldots,1),n}^{(\nu-m)},$$
where $L_{\emptyset, (1,\ldots,1),n}^{(\nu-m)}$ denotes the exceptional polynomial associated to the partitions $\emptyset, (1,\ldots,1)$, see \cite[Def. 4]{ArnoN}, and $c_{n,m}$ is a constant depending on $n,m, \nu.$

\end{rmk}

\section{Zeros of matrix exceptional Laguerre polynomials}\label{sec:numerics}
In this section we present some numerical information on the zeros of the 
matrix exceptional Laguerre polynomials discussed in Part \ref{part:matriXOL}. 
The zeros are the zeros of the determinant of the matrix exceptional 
polynomials, as for the matrix valued orthogonal polynomials \cite{DuranMarkov}, \cite{DuranL}. 

Our numerical experiments lead to the following conjectures for the zeros of 
$\det(\widehat{P}_n^{(\alpha,\nu,m)})$ under the 
condition $\nu >\max(0,m-1)$:
\begin{enumerate}
    \item $\det(\widehat{P}_n^{(\alpha,\nu,m)})$ has $nN$ real simple zeros;
    \item $\det(\widehat{P}_n^{(\alpha,\nu,m)})$ has $m$ clusters of $N^2$ complex zeros outside $[0,\infty)$, and the multiplicity can be more than one;
    \item higher multiple complex zeros of $\det(\widehat{P}_n^{(\alpha,\nu,m)})$ coincide with the zeros of $\Up_m$ and there are $mN$ multiple complex zeros of multiplicity $N-1$.  
\end{enumerate}
\begin{figure}[h]
\subfloat[$N=2$, $m=30$, $n=7$, $\al=30$, $\nu=31.$]{\includegraphics[width = 0.3\textwidth]{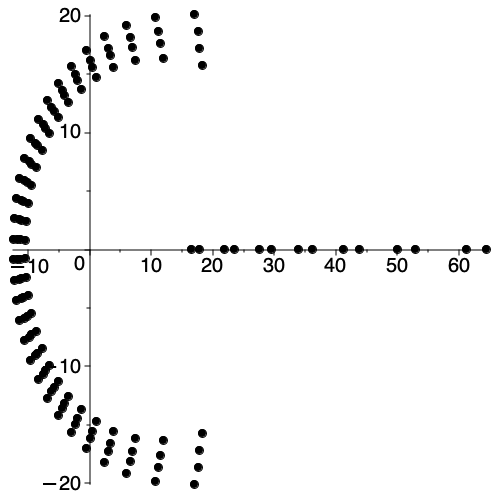}\quad} 
\subfloat[$N=3$, $m=13$, $n=5$, $\al=14$, $\nu=14$. The solid circles are double zeros.]{\includegraphics[width = 0.3\textwidth]{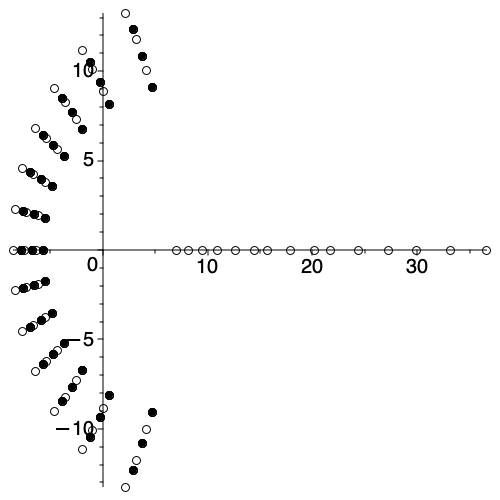}\quad} 
\subfloat[$N=2$, $m=30$, $n=7$, $\al=30$, $\nu=27.5$.]{\includegraphics[width = 0.3\textwidth]{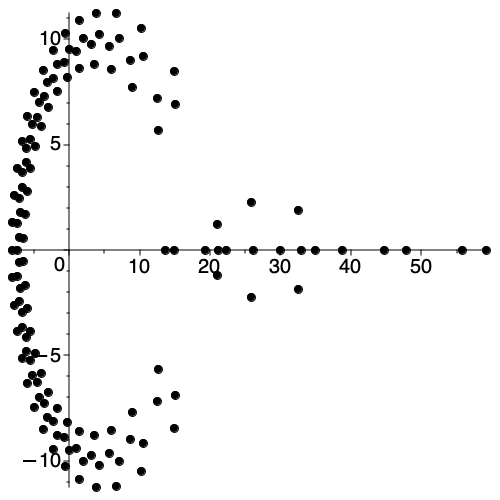}}\\
\caption{Numerical support for the conjectures in Section \ref{sec:numerics} }
\label{fig:plotzeroes}
\end{figure}
We support this with plots in Figure \ref{fig:plotzeroes}. We observe that, for $N=1$, the first conjecture was discussed in \cite[\S 6.1]{ArnoN} and also in \cite[\S 4.3]{MarcellanU}.
The zeros of $\det(\widehat{P}_n^{(\alpha,\nu,m)})$ can also be generated for several degrees relaxing the condition 
$\nu>\max(0,m-1)$. If we drop this condition, the behaviour of the zeros gets more chaotic and the conjectures do not seem valid.

\begin{appendices}

\section{Matrix differential equations with singularities}\label{sec:appA:MDEsingularities}

We discuss some results on the notion of singularities and regular singularities of matrix differential equations following Coddington and Levinson \cite{CoddL}, especially Chapter 4 and Exercises 2 and 13. The approach of Coddington and Levinson \cite{CoddL} is very general, and for the Frobenius method for second order differential equations and the indicial equation,  see Coddington and Levinson \cite[Ch.~4, \S 7-8]{CoddL} and e.g. Temme \cite[\S 4.2]{Temm}. For the special case of a matrix differential operator of hypergeometric type, see Tirao \cite{Tira}. In the appendix we follow the convention of operators acting from the left as in \cite{CoddL}. Let us consider, see \cite[Ch.~4, Exercise 2]{CoddL} for $n=2$,
\begin{equation}\label{eq:A-2ndorderMVde}
z^2 \, w''(z) + z B_1(z) \, w'(z) + B_2(z)\, w(z) = 0,
\end{equation}
where $B_1$ and $B_2$ are $N\times N$-matrix valued functions which are analytic in 
an open disk of radius $r>0$, meaning that each entry is analytic. We look for vector valued solutions $w$ taking values in $\C^N$ of \eqref{eq:A-2ndorderMVde}. 
Following \cite[Ch.~4, \S~5]{CoddL} we put $\vp_1(z)=w(z)$, $\vp_2(z)=z w'(z)$ and 
the $\C^{2N}$ valued function $\vp = \begin{pmatrix} \vp_1\\ \vp_2 \end{pmatrix}$ satisfies
\begin{equation}\label{eq:A-1storderreformulation}
\vp'(z) = A(z) \vp(z), \qquad A(z) = \frac{1}{z} 
\begin{pmatrix} 0 & \Id \\ -B_2(z) & \Id - B_1(z) \end{pmatrix}
\end{equation}
where $A$ is $2N\times 2N$-matrix, which is a $2\times 2$-block matrix of $N\times N$-matrices, where $0$, respectively $\Id$, denotes the zero $N\times N$-matrix, respectively the identity matrix.  In particular, $z=0$ is a singularity of the first kind and a regular singularity, see 
\cite[Ch.~4, \S~2]{CoddL}, and this is in line with \cite{Tira} for the case of the matrix hypergeometric differential operator. 
Putting $R= zA(z)\vert_{z=0} = \begin{pmatrix} 0 & \Id \\ -B_2(0) & \Id - B_1(0) \end{pmatrix}$ the corresponding indicial equation is $\det(\mu-R)=0$, see \cite[p.~127]{CoddL}, and using \cite[Thm.~3, (14)]{Silv} the 
indicial equation for \eqref{eq:A-1storderreformulation} becomes 
\begin{equation}\label{eq:A-indicialeq}
\det\bigl(\mu (\mu-1)+\mu B_1(0) + B_2(0)\bigr)= 0.
\end{equation}
The solutions of the indicial equation are the exponents of \eqref{eq:A-1storderreformulation} and \eqref{eq:A-2ndorderMVde}. 
Since $B_1$ and $B_2$ are analytic, the matrix $A(z)=\sum_{m=-1}^\infty z^mA_m$ is a convergent Laurent series around $z=0$, hence \cite[Thm.~3.1]{CoddL} shows that any formal solution to \eqref{eq:A-1storderreformulation} is convergent in a region $0<|z|<a$ for some $a>0$. Moreover,
\eqref{eq:A-1storderreformulation} has a fundamental matrix $\Phi$ of solutions of the form 
$\Phi(z) = z^{\hat{R}} P(z)$ for some region $0<|z|<a$ for some $a>0$. Here $P(z)$ is a convergent power series and $\hat{R}$ is related to $R$, see \cite[Ch.~4, Thm.~4.2]{CoddL}. 
In case $R$ is semisimple, there is a basis of solutions $\vp_\la(z) = z^\la f(z)$, 
$\la$ eigenvalue of $R$, and $f$ a vector valued analytic function in a disk $|z|<r$ for some $r>0$, see Exercise 13 of \cite[Ch.~4]{CoddL} where also the case of more general $R$ is considered.  
In particular, we have $2N$ solutions to \eqref{eq:A-2ndorderMVde} of the form 
$z^\mu f(z)$ for $\mu$ solving \eqref{eq:A-indicialeq} assuming $R$ is semisimple. Assuming 
moreover that $B_2(0)=0$ as matrices, \eqref{eq:A-indicialeq} reduces to 
$\mu^N \det(\mu-1+B_1(0))=0$ and $B_1(0)$ being semisimple. In this case we have 
$N$ linearly independent analytic solutions in a disk centered at $0$, and we have $N$ solutions of the form $z^{1-\mu_i} f(z)$, $f$ analytic in a disk centered at $0$, for $\{\mu_i\}_{i=1}^N$
the set of eigenvalues of $B_1(0)$.

Note that the same analysis can be done for any other point $z_0\in \C$ by the change of coordinates $z-z_0$ as well as for $z_0=\infty$ by replacing $z=1/\zeta$, see \cite{CoddL}. 

\end{appendices}



\end{document}